\documentclass{amsart}
\usepackage{amssymb,amscd,verbatim,latexsym}
\usepackage{enumerate}
\usepackage[mathscr]{eucal}
\usepackage{epsfig,layout,graphics}
\usepackage[dvipsnames,usenames]{color}

\newcommand\datver[1]{\def\datverp
{\par\boxed{\boxed{\text{Version: #1; Run: \today}}}}}
\datver{1.0; Revised: 03/10/2014 by V.}
\renewcommand\datverp{~}

\newcommand\seq{ \, = \, }
\newcommand\ede{ \, := \, }
\newcommand\weight{\langle x \rangle}
\newcommand\ad{\operatorname{ad}}

\newcommand\pa{{\partial}}
\newcommand\adj{\operatorname{ad}}

\newcommand{\Cf}{\mathfrak{C}}
\newcommand{\Df}{\mathfrak{D}}
\newcommand{\Ef}{\mathfrak{E}}

\newcommand{\CC}{\mathbb C}

\newcommand{\RR}{\mathbb R}

\newcommand{\ZZ}{\mathbb Z}

\newcommand{\maC}{\mathcal C}

\newcommand{\maE}{\mathcal E}
\newcommand{\maF}{\mathcal F}
\newcommand{\maH}{\mathcal H}
\newcommand{\maK}{\mathcal K}
\newcommand{\maL}{\mathcal L}
\newcommand{\maP}{\mathcal P}

\newcommand{\maW}{\mathcal W}

\newtheorem{theorem}{Theorem}[section]
\newtheorem{proposition}[theorem]{Proposition}
\newtheorem{corollary}[theorem]{Corollary}

\newtheorem{lemma}[theorem]{Lemma}

\theoremstyle{definition}
\newtheorem{definition}[theorem]{Definition}
\theoremstyle{remark}
\newtheorem{remark}[theorem]{Remark}

\usepackage{xcolor}                                           %
\title[SABR]{Heat kernels, solvable
Lie groups, and the mean reverting SABR
stochastic volatility model}

\author{Siyan Zhang} \address{Pennsylvania State University,
  Mathematics Department, University Park, PA 16802, USA}
\email{zhang\_s@math.psu.edu}

\author{Anna L. Mazzucato}
\address{Pennsylvania State
    University, Mathematics Department, University Park, PA 16802, USA}
\email{alm24@psu.edu}

\author{Victor Nistor} \address{Universit\'e de
  Lorraine, UFR MIM, Ile du Saulcy, CS 50128, 57045 METZ Cedex 01,
  France and Pennsylvania State
    University, Mathematics Department, University Park, PA 16802, USA}
\email{victor.nistor@univ-lorraine.fr}

\date{\today}

\begin{document}

\begin{abstract}
We use commutator techniques and calculations in solvable Lie groups
to investigate certain evolution Partial
Differential Equations (PDEs
for short) that arise in the study of stochastic volatility models for
pricing contingent claims on risky assets.  In particular, by
restricting to domains of bounded volatility, we establish the
existence of the semi-groups generated by the spatial part of the
operators in these models, concentrating on those arising in the so-called
``SABR stochastic volatility model with mean reversion.''  The
main goal of this work is to approximate the solutions of the Cauchy
problem for the SABR PDE with mean reversion, a parabolic problem the
generator of which is denoted by $L$.  The fundamental solution for
this problem is not known in closed form. We obtain an approximate
solution by performing an expansion in the so-called {\em volvol} or
volatility of the volatility, which leads us to study a degenerate
elliptic operator $L_0$, corresponding the the zero-volvol case of the
SABR model with mean reversion, to which the classical results do not
apply. However, using Lie algebra techniques we are able to derive an
exact formula for the solution operator of the PDE \ $\pa_t u - L_0 u
= 0$.  We then compare the semi-group generated by $L$--the existence
of which does follows from standard arguments--to that generated by
$L_0$, thus establishing a perturbation result that is useful for
numerical methods for the SABR PDE with mean reversion. In the
process, we are led to study semigroups arising from both a strongly
parabolic and a hyperbolic problem.
\end{abstract}



\maketitle

\tableofcontents

\section{Introduction}

We study certain parabolic partial differential equations (PDEs for
short) that arise in the study of stochastic volatility models for
pricing contingent claims on risky assets.  More specifically, we
consider the PDE
\begin{equation}\label{eq.def.lSABR}
	\pa_t u - L u \ede \pa_t u -
        \kappa(\theta-\sigma)\pa_{\sigma}u - \frac{\sigma^2}{2}\big [
          (\pa_x^2u-\pa_xu) - \nu\rho\pa_x\pa_{\sigma} u -
          \frac{\nu^2}{2}\pa_{\sigma}^2 u \big ] \,= \, 0 \,
\end{equation}
for the function $u(t,\sigma, x)$, where $t\geq 0$, $\sigma > 0$ and
$x \in \RR$.  This equation is a forward Kolmogorov equation for the
probability density function associated to a two-dimensional
stochastic process for the variables $\sigma$ and $x$. The parameter
$\theta>0$ represents the mean of the $\sigma$ process, $\kappa>0$ is
a parameter measuring the strength of the mean reversion, $\nu>0$ is
the variance of the $\sigma$ process, and $\rho$ measures the
correlation between the $x$ and the $\sigma$ processes. This PDE is
often called the {\em $\lambda$SABR PDE} and has recently received
attention in the literature due to its applications in pricing options
in mathematical finance and financial applications
\cite{Lesniewski02, Lesniewski15}, where it is used as an alternative
to the Black-Scholes PDE. In this context, the Green's function is
called the {\em pricing kernel of the economy}, $x$ represents
the price of an underlying risky asset such as a stock, and $\sigma$
is its volatility. Thus $\sigma$ itself follows a stochastic process,
hence the $\lambda$SABR model is a stochastic volatility model in which
  $\nu$ represents the volatility of the volatility or {\em volvol}.
Stochastic volatility models are known to perform better
in practice than the Black-Scholes model (see
e.g. \cite{Lesniewski02, Heston, HullWhite}).

Our method consists in the following decomposition of the operator $L$:
\begin{equation}\label{eq.decomp.L}
 L \seq A + \frac{\sigma^{2}}{2} B + \nu L_1 + \nu^{2} L_2\,,
\end{equation}
where
\begin{equation}\label{def.operators}
\begin{gathered}
  A \ede \kappa (\theta - \sigma) \pa_\sigma\,, \quad B \ede
  \pa_x^2-\pa_x \,,\\
  L_1 \ede \rho\sigma^2\pa_x\pa_\sigma \, , \ \ \mbox{ and } \ \
  L_2 \ede \frac12\sigma^2\pa_\sigma^2 \,,
\end{gathered}
\end{equation}
and then in  studying separately these operators and their combinations, based
on the commutator identities that they satisfy. We thus establish that
$L$, $A$, $B$, and
\begin{equation}\label{eq.def.L0}
 L_0 := A + \frac{\sigma^{2}}{2} B
\end{equation}
generate strongly continuous or $c_0$ semi-groups, provided that we
restrict to a domain of bounded volatility $\sigma \in I := (\alpha,
\beta)$, where $0 < \alpha < \theta < \beta < \infty$.  We stress that
$L_0$ is a degenerate operator, in the sense that the diffusion matrix
associated to $L_0$ is not full rank. Therefore, the existence of the
semigroup does not follows from standard arguments.

Throughout, if $T$ is a linear operator that generates a semigroup, we
shall denote such semigroup by the usual notation $e^{t T}$, $t\geq
0$.

We will obtain explicit formulas for the kernel of the semi-groups
generated by $A$, $B$, and $L_0$.  While we have no explicit formulas
for the kernel of $e^{tL}$, the solution operator of the PDE
\eqref{eq.def.lSABR} of interest in applications, we are nonetheless
able to estimate the difference $e^{tL}h -e^{tL_0}h$, provided
$h(\sigma,x)$ has enough regularity in $\sigma$. The function $h$
represents the initial data for the Cauchy problem associated to
\eqref{eq.def.lSABR}, and in the specific applications we have in
mind, it is actually an analytic, or even constant, function in
$\sigma$.

The semi-groups investigated in this paper will typically act on
exponentially weighted Sobolev spaces. The reason for considering
exponentially weighted spaces is that, in the applications of
interest, the initial data $h$ for \eqref{eq.def.lSABR} is of the form
$h(\sigma, x) := |e^{x} - K|_{+}$, where $|y|_{+} = (y)^{+} :=(y +
|y|)/2$ denotes the {\em positive part} of the number $y \in
\RR$. This particular type of initial data arises in pricing of
so-called European call options (we refer to \cite{Gatheral:xx,
  Shreve} for a more detailed discussion of options).  The practical
meaning of the initial condition $h$ is the {\em payoff} of the option
at maturity.  Similar initial conditions are used for other types of
options, such as American and Asian options. From a mathematical point
of view, the form of $h$ requires exponential weights and implies low
regularity of the initial data in the $x$ direction, but provides
analytic regularity in the $\sigma$ direction, which we indeed exploit
in our estimate of $e^{tL}h - e^{tL_0}h$ (see Equation
\eqref{eq.intro.error} below and the statement of one of our main
results, Theorem \ref{thm.estimate}).

The semi-groups generated by the operators $A$ and $B$, and $L$ can be
obtained using classical methods, since the operator $A$ gives rise to
a transport evolution equation, whereas $B$ and $L$ are uniformly
strongly elliptic. In particular, we show that $B$ and $L$ generate
{\em analytic semi-groups}.  However, as already mentioned, classical
methods do not apply to $L_0$, which is degenerate.
We will employ a
different strategy, which allows us to establish the generation of
$c_0$ semigroup by $L_0$ and obtain an explicit formula for its
kernel. The key observation is that the operators $A$ and
$\frac{\sigma^{2}}{2}B$ generate a solvable, finite-dimensional Lie
algebra.

Having an explicit formula is important in obtaining an accurate, yet
easily computable, approximation of the solution operator $e^{t L}$,
one of the main goals of this work. To this end, we derive an error
estimate of the form:
\begin{equation}\label{eq.intro.error}
 \|e^{tL} h - e^{tL_0} h\|_{L^2} \le C \nu \big ( \|\pa_\sigma
 h\|_{L^2} + \|h\|_{L^2} \big ) \,,
\end{equation}
for $\nu \in (0, 1]$ and with a constant $C$, possibly dependent on
$L$ and $\kappa$, but not on $h$ and $\nu$ (see Theorem
\ref{thm.estimate} for a complete statement). In the process, we also
establish several mapping properties for the semi-groups generated by
$L_0$ and $L$.  The method of proof is a perturbative argument based
on heat kernels estimates, following the method developed in
\cite{Wen1, Wen2}. This method extends the work on Henry-Labord\`ere
on heat kernel asymptotics \cite{Labordere07, LabordereBook}. A
similar method was developed by Pascucci and his collaborators
\cite{PascucciCEJM, Pascucci13}. Heat kernel asymptotics were employed
in this context also by Gatheral and his collaborators
\cite{Gatheral12a, Gatheral12b}.  See also \cite{choulli, Lorig15, Pascucci15,
  Nakagawa, Lenotre, FeehanPop3}. We also mention that fundamental solutions for degenerate equations
related to \ $\partial_t u - L_0 u=F$, but in the context of ultraparabolic
equations satisfying H\"ormader's  conditions for hypoellipticity, which does
not hold for $\partial_t -L_0$, have been studied by many authors, starting
with the seminal work of Kolmogorov \cite{K34} (see
\cite{FeehanPop2, FeehanPop1,
PascucciUltraparab,UltraparabPreprint} for some recent, relevant works).

The paper is organized as follows. In Section \ref{sec2} we review a
few needed facts on evolution equations and semi-groups of
operators. We also introduced the exponentially weighted spaces used
in this paper. In Section \ref{sec3}, we show that the operators $L$
and $B$, which are both strongly parabolic, generate analytic
semi-groups on weighted spaces, using the Lumer--Phillips theorem and
the results of Section \ref{sec2}. Section \ref{sec4} deals with the
semi-groups generated by $A$, which is of transport type, and $L_0$,
which is degenerate parabolic. An explicit formula for $e^{tL_0}$ is
obtained by combining the results for the operators $A$ and $B$, more
specifically by exploiting the commutator identities that $A$ and
$f(\sigma) B$ satisfy and Lie group ideas. The last section, Section
\ref{sec5}, contains some additional results: a more detailed
discussion of Lie group ideas in evolution equations and the proof of
the error estimate \eqref{eq.intro.error}.

\subsection*{Notation:}
We close this Introduction with some notation used throughout. By
$\|\cdot\|$ we denote the functional norm in a Banach space, while the
norm of finite-dimensional vectors in $\RR^n$ will be simply denoted
by $|\cdot |$. Lastly, by $(,)$ we mean either the pairing between a
Banach spaces and its dual, or the $L^2$ inner product, depending on
the context.

\subsection*{Acknowledgments:}
The first and second authors were partially supported by the US National
Science Foundation grant DMS 1312727. The third author was supported
by France {\em Agence Nationale de la Rech\'erche} ANR-14-CE25-0012-01
(SINGSTAR).

\section{One parameter semi-groups}
\label{sec2}

This section is devoted to survey general facts about abstract
evolution equations and semi-groups of operators. We also review
needed facts about the function spaces we employ, in particular {\em
  exponentially weighted Sobolev spaces}. As remarked in the
Introduction, these spaces are needed to handle initial conditions of
the form $h(\sigma, x) := |e^{x} - K|_{+}$, $(\sigma, x) \in (0,
\infty) \times \RR$.  Most of the results presented in this section
are known. We follow primarily,\cite{AmannBook, Lunardi, PazyBook}.

\subsection{Unbounded operators and $c_0$ semi-groups}

We begin by recalling the notion of a semi-group generated by a linear
operator.  Throughout, {\em $\maL(X)$ will denote the space of bounded
  linear operators on a Banach space $X$}, which is a Banach algebra
using the operator norm.

\begin{definition} 
Let $X$ be a Banach space. A {\em strongly continuous} or {\em $c_0$
  semi-group of operators on $X$} is a family of bounded operators
$S(t) : X \to X$, $t \ge 0$, satisfying:
\begin{enumerate}[(i)]   
\item $S(t_1 + t_2) = S(t_1) S(t_2)$, for all $t_i \ge 0$,
\item $S(0)=I$, where $I$ represent the identity operator on $X$,
\item $\lim_{t \to 0} S(t)x = x$, for all $x \in X$, where the limit
  is taken with respect to the topology of $X$.
\end{enumerate}
\end{definition}

We recall that a function $T : [a, b] \to \maL(X)$ is {\em strongly
  continuous} if the map $[a, b] \ni t \to T(t) \xi \in X$ is
continuous for every $\xi \in X$.  It follows from the definition of a
$c_0$ semi-group and the Banach-Steinhaus theorem that, if $S(t)$ is a
$c_0$ semi-group of operators on $X$, then $S(t)$ is strongly
continuous in $t$, hence the name strongly continuous semigroups.

We shall need also the notion of analytic semi-groups.
To this end, for a given $\delta>0$, we let
\begin{equation}\label{eq.def.Delta.delta}
 \Delta_{\delta} \ede \{\, z = r e^{\imath \theta}\,, \ - \delta <
 \theta < \delta ,\ r > 0 \,\}\,.
\end{equation}

\begin{definition} 
Let $X$ be a Banach space. An {\em analytic semi-group of operators on
  $X$} is a function $S : \Delta_{\delta} \cup \{0\} \to \maL(X)$,
$\delta > 0$, with the properties
 \begin{enumerate}[(i)] 
  \item $S$ is analytic in $\Delta_{\delta}$;
  \item $S(z_1 + z_2) = S(z_1) S(z_2)$, if $z_i \in \Delta_\delta \cup
    \{0\}$;
    \item $S(0)=I$, the identity operator on $X$;
  \item $\lim_{z \to 0} S(z)x = x$, for all $x \in X$.
 \end{enumerate}
\end{definition}

The limit $\lim_{z \to 0} S(z)x $ is computed for $z \in
\Delta_{\delta}$.  An analytic semi-group is, in particular,
a $c_0$ semi-group.

\begin{definition}\label{def.unbounded}
 Let $X$ be a normed space. A {\em (possibly unbounded) linear
   operator on $X$} is a linear map $T : D(T) \to X$, where $D(T)
 \subset X$ is a linear subspace, called the {\em domain} of $T$.  We
 say that $T$ is {\em closed} if its graph is closed.
\end{definition}

Unbounded linear operators arise naturally as the generators of $c_0$
semi-groups.

\begin{definition}\label{def.gen.sgr}
The {\em generator} $T$ of a $c_0$ semi-group $S(t)$ on $X$ is the
operator $T\xi := \lim_{t \searrow 0}\, t^{-1} \big (S(t) \xi - \xi
\big)$, with domain the set of vectors $\xi \in X$ for which the limit
exists.
\end{definition}

It is known that the generator of a $c_0$ semi-group is closed and
densely defined. We next review  criteria for an {\em unbounded}
operator $T$ to generate a $c_0$ semi-group $S(t)$.  Then $u(t)
:= S(t) h$ is a (suitable) solution of $u' - Tu =0$, $u(0) = h$. A
useful criterion for $T$ to generate a $c_0$ semi-group is provided by
the Lumer-Phillips theorem, which we discuss next.  Since two $c_0$
semi-groups with the same generator coincide \cite{AmannBook,
  PazyBook}, we shall write $S(t) = e^{tT}$ for the semi-group
generated by $T$, if such a semi-group exists.

\subsection{Dissipativity}
In the following, $\Re(z) = \Re z$ will denote the real part of $z \in
\CC$.  Let $X$ be a Banach space and let $X^{*}$ denote its dual. If
$x \in X$, the Hahn-Banach theorem implies, in particular, that the
set
\begin{equation*}
 \maF(x) \ede \{ f \in X^{*}, f(x) = \|x\|^{2} = \|f\|^{2} \}
\end{equation*}
is not empty.

\begin{definition}\label{def.qd}
A (possibly unbounded) operator $T$ on a Banach space $X$ is called
{\em quasi-dissipative} if there exists $\mu \ge 0$ such that, for
every $x \in D(T)$, there exists an $f \in \maF(x) \subset X^{*}$ with
the property that and $\Re \big ( f(Tx - \mu x) \big ) \le 0$.
\end{definition}

This definition is simply saying that for some $\mu>0$, the operator
$Tx-\mu x$ is dissipative.

The {\em numerical range} of $T$, denoted $\mathfrak{N}(T)$, is the
set
\begin{equation}\label{eq.num.r}
 \mathfrak{N}(T) \ede \{\, f(Tx),\ \|x\| = 1, f \in \maF(x)\,\} \,.
\end{equation}
A quasi-dissipative operator $T$ is thus one that has the property
that
\begin{equation}\label{eq.r.d}
 \mathfrak{N}(T) \, \subset \, \{\, z \in \CC,\ \Re(z) \le \mu \, \}
 \seq \mu + \Delta_{\pi/2}^{c}\,
\end{equation}
with $\Delta_\delta$ defined in Equation \eqref{eq.def.Delta.delta}
and $\Delta_\delta^{c} := \CC \smallsetminus \Delta_\delta$ its
complement.


Quasi-dissipativity, together with some mild conditions on the
operator $T$ stated below, is sufficient for the generation of a $c_0$
semigroup, by the celebrated Lumer-Phillips theorem, which we now
recall for the benefit of the reader \cite{AmannBook, PazyBook}.

\begin{theorem}[Lumer-Phillips]\label{thm.LumerPhillips}
 Let $X$ be a Banach space and let $T$ be a densely defined,
 quasi-dissipative operator on $X$ such that $T-\lambda$ is invertible
 for $\lambda $ large.  Then $T$ generates a $c_0$ semi-group on $X$.
\end{theorem}

By strengthening the condition \eqref{eq.r.d}, we obtain the following
similar theorem that yields generators of {\em analytic
  semi-groups}. The proof of this theorem is contained in the proof of
Theorem 7.2.7 in \cite{PazyBook}.

\begin{theorem}\label{thm.analytic}
 Let $X$ be a Banach space and let $T$ be a densely defined operator on
 $X$ such that $\mathfrak{N}(T) \subset \mu + \Delta_{\vartheta}^{c}$
 for some $\mu \in \RR$ and some $\vartheta > \pi/2$. Assume also that
 $T-\lambda$ is invertible for $\lambda$ large. Then $T$ generates an
 analytic semi-group.
\end{theorem}

We note that the assumption that $T-\lambda$ be invertible in Theorem
\ref{thm.analytic} implies that $T$ is closed.  The theorem is
especially useful when $T$ is a uniformly strongly elliptic operator
(see Definition \ref{def.u.s.e}) in view of the following Lemma, the
proof of which is again contained in the proof of Theorem 7.2.7 in
\cite{PazyBook}. See also \cite{ArendtForm, LionsBook61, KatoPerturbation}.

\begin{lemma}\label{lemma.n.r}
 Let $P$ be an order $2m$ differential operator on some domain $\Omega
 \subset \RR^{n}$, regarded as an unbounded operator on $L^2(\Omega)$
 with domain $D(P) \subset H^{2m}(\Omega)$.  We assume that there
 exists $C>0$ such that
 \begin{equation*}
   \Re(Pv, v) \le - C^{-1} \|v\|_{H^m(\Omega)}  \ \ \mbox{ and } \ \
  |(Pv, v)| \le C \|v\|_{H^m(\Omega)}\,, \quad (\forall)\, v \in D(P)\,.
 \end{equation*}
 Then $\mathfrak{N}(P) \subset \Delta_{\vartheta}^{c}$ for some
 $\vartheta > \pi/2$.
\end{lemma}

From Theorem \ref{thm.analytic} and Lemma \ref{lemma.n.r}, we get the
following corollary.

\begin{corollary}\label{cor.analytic}
 Let $P$ be as in Lemma \ref{lemma.n.r} and assume that $D(P)$ is
 dense in $L^{2}(\Omega)$ and that $P-\lambda$ is invertible for
 $\lambda$ large. Then $P$ generates an analytic semi-group on $X$.
\end{corollary}

\subsection{Classical and other types of solutions}
Let us consider the initial-value problem for abstract parabolic
equations of the form
\begin{equation}\label{eq.evolution.form}
 \pa_t u - P u = F\,, \quad u(0) = h\in X \,,
\end{equation}
where $P$ is a (usually unbounded) operator on a Banach space $X$
and with domain $D(P)$. In our applications, $X$ will be a space of
functions on $\Omega$, but first
we consider this equation abstractly, from the point of view of
semi-groups of operators.

\begin{definition}\label{def.strong.s}
We shall say that a function $u : [0,T]  \to X$ is a
  {\em strong solution} of the initial value problem
  \eqref{eq.evolution.form} for $F \in \maC([0, T]; X)$ if
\begin{enumerate}[(i)] 
\item $u$ is continuous for the norm topology on $X$ and $u(0)=h$;
\item $\pa_t u = u'$ is defined and continuous as a function $(0,T] \to X$;
\item $u(t) \in D(P)$ for $t \in (0,T]$; and
\item $u$ satisfies the equation $\pa_t u(t) - P u(t) = F(t) \in X$,
  for $t \in (0, T]$.
\end{enumerate}
\end{definition}

We shall need also the following weaker form of a solution.

\begin{definition} \label{def.mild.s}
A function $u: [0,T]\to X$ is called a {\em mild solution} of the
  initial-value problem \eqref{eq.evolution.form} if $h\in X$, $F\in
  L^1([0,T],X)$, and
\begin{equation*}
    u(t) \, = \, e^{t P} h +\int_0^t e^{(t-\tau) P}\, F(\tau)\,d\tau,
\end{equation*}
with equality as elements of $X$ pointwise in time $t \in (0, T)$.
\end{definition}

The following remark recalls the connection between semi-groups and the various
types of solutions of the Initial Value Problem \eqref{eq.evolution.form}.

\begin{remark}\label{rem.solutions}
For the applications of interest in this work, we can reduce to homogeneous
equations, that is $F(0)=0$, as we assume now. We also assume
that the operator $P$ generates a $c_0$ semi-group $e^{tP}$ on $X$. Then $u(t)
:= e^{tP} h$ is a mild solution for any $h \in X$. If, moreover, $h \in D(P)$
or if $P$ generates an analytic semi-group, then $u (t) := e^{tP} h$ is also a
strong solution of Equation \eqref{eq.evolution.form} (see
\cite{AmannBook, Lunardi, PazyBook}, for instance).
\end{remark}

We are interested in the case when $P$ is a $m$-th order partial differential operator
defined on a domain $\Omega \subset \RR^d$:
\begin{equation}\label{eq.form.P}
P \ede \sum_{|\alpha| \le m} a_{\alpha} \pa^{\alpha}\,,
\end{equation}
with coefficients $a_\alpha \in C^\infty (\overline{\Omega})$. We
shall occasionally use the convenient notation:
\begin{equation*}
    u(t)(q) \ede u(t,q) \,, \quad t \ge 0\ \mbox{ and } \ q \in
    \Omega\,,
\end{equation*}
which is in agreement with \eqref{eq.evolution.form}. When $P = L$,
acting on $L^2(\Omega)$, $\Omega = (0, \infty) \times \RR$, $F = 0$,
and $h(\sigma, x) := |e^{x} - K|_{+}$, we recover the initial-value
problem \eqref{eq.def.lSABR}. In that case, we are interested in
{\em classical} and {\em weak} solutions. We assume that $X$
is a space of functions on $X$, that is, $X \subset L^{1}_{loc}(\Omega)$.
We also assume that the domain of $P$ contains the space of smooth
functions with compact support in $\Omega$, and hence the same is satisfied
by its adjoint.

\begin{definition}\label{def.classical.s}
We shall say that a function $u : [0,T] \times \Omega \to \CC$ is a
  {\em classical solution} of the initial value problem
  \eqref{eq.evolution.form} if
\begin{enumerate}[(i)] 
\item $u$ is continuous on $[0,T] \times \Omega$ and $u(0,q)=h(q)$, for
  all $q\in \Omega$;
\item $\pa_t u = u'$ and $\pa^{\alpha}u$, $|\alpha|\leq m$, are
  defined and continuous on $(0,T]\times \Omega$; and
\item $u$ satisfies the equation $\pa_t u- P u = F$ pointwise in $(0,
  T] \times \Omega$.
\end{enumerate}
If boundary conditions for $u$ on $\pa \Omega$ are given, we require
them to be satisfied as equalities of continuous functions.
\end{definition}

It follows that if $u$ is a classical solution, then $F$ is continuous.
We note that in the abstract setting, strong solutions are often referred to
as classical solutions (see e.g. \cite{PazyBook}.

%
\begin{remark}\label{rem.a.est}
 We recall that, if $T$ is the generator of an analytic semi-group
 $e^{tT}$ on a Banach space $X$, then $T^{n}e^{tT}$ extends to a
 bounded operator on $X$ and there exists $C > 0$ such that
\begin{equation}\label{eq.regularity}
 \|T^{n}e^{tT}\| \, \le \, Ct^{-n} \,, \ \ \mbox{ for all } \ t \in
 (0, 1]\,.
\end{equation}
\end{remark}

The following lemma follows from  known results  (cf.
\cite{Lunardi,PazyBook}).

\begin{lemma}\label{lemma.classical.s}
 Assume that there exists $n \ge 0$ such that $D(P^{n}) \ni f \to
\pa^{\alpha} f \in
 \maC (\overline{\Omega})$ is continuous for all $|\alpha| \le m$.
 In addition, assume that $P$ generates a $c_0$ semi-group on $X$
 and that $F = 0$.
 Then $u(t) := e^{tP}h$ is a classical solution of Equation \eqref{eq.evolution.form}
 for all $h \in D(P^{n+1})$.
\end{lemma}

\begin{proof}
 For each fixed $t$, $u(t) \in D(P^{n+1})$ defines a continuous
 function on $\overline{\Omega}$, since $D(P^{n}) \subset \maC
 (\overline{\Omega})$ continuously.
 The same argument shows that the map $[0, T] \ni t \to u(t) \in \maC
(\overline{\Omega})$
 is continuous, and hence $u$ is continuous on $[0, T] \times \Omega$, that
$\pa_t u$, $\pa^{\alpha}u$, $|\alpha|\leq m$, are
  defined and continuous on $(0,T]\times \Omega$, and that $u' = Pu$ (this is
where we need the stronger
  assumption that $h \in D(P^{n+1})$, since we need $\epsilon^{-1}(u(t+\epsilon) - u(t)) \to
  Pu(t) \in D(P^{n})$, as $\epsilon \to 0$).
\end{proof}

Let us denote by
\begin{equation}\label{eq.def.transpose}
  P^tv := \sum_{|\alpha| \le m} (-1)^{|\alpha|} \pa^{\alpha}
  (a_{\alpha}v)
\end{equation}
be the {\em transpose} of $P$ (so that $\int_{\Omega} (Pu)v dx=
\int_{\Omega} u(P^t v) dx$ whenever $u$ and $v$ are compactly
supported in $\Omega$). Similarly, we then have the following
definition of {\em weak} or distributional solutions.

\begin{definition}\label{def.weak.s}
We shall say that $u : [0,T)\times \Omega \to \CC$ is a
  {\em weak solution} of the initial value problem
  \eqref{eq.evolution.form} if $u, F \in L^{1}_{loc}([0,T) \times
    \Omega)$ and, for all $\phi\in \maC_c^\infty([0,T)\times \Omega)$,
\begin{equation}\label{eq.weak.s}
 \int_{\Omega} \, \Big [ \, \phi(0, x)h(x) \, + \, \int_0^T \, \big (
   \pa_t \phi \, + \, P^t \phi \big )\, u \, dt \, \, + \, \int_0^T \,
   \phi F \, dt \, \Big ] \, dx \, = \, 0 \,.
\end{equation}
If, moreover, $v$ is also a classical solution on $[\delta, T]$
for all $\delta > 0$, we shall say that $v$ is a {\em classical} solution
on $(0, T]$. If $v$ is a classical solution on $(0, T]$, for all
$T< R$, then we say that $v$ is a classical solution on $(0, R)$.
\end{definition}

Again, the following lemma is well-known (see e.g. \cite{PazyBook}).

\begin{lemma}\label{lemma.weak}
 Assume that $P$ generates a $c_0$ semi-group on $X$. Then $u(t) := e^{tP}h$ is
a weak solution of the homogeneous Initial-Value
Problem \eqref{eq.evolution.form} with $F=0$ for all $h \in X$.
\end{lemma}

\begin{proof}
We assume that $h \in D(P)$ and that $\phi$ is as in Definition
 \ref{def.weak.s}. Then the function $\psi(t) := (e^{tP}h, \phi(t))$
 is continuously differentiable on $[0, T]$. The relation
 $\psi(T) - \psi(0) = \int_{0}^{T} \psi'(t)dt$ gives that $u(t) := e^{tP}h$
 is a weak solution of the IVP \eqref{eq.evolution.form}. Since
 the weak form \eqref{eq.weak.s} depends continuously on $h$, we obtain that
 $u(t) := e^{tP}h$ is a weak solution of \eqref{eq.evolution.form}
 by the density of $D(P)$ in $X$.
\end{proof}

Combining the two lemmas above we obtain.

\begin{proposition}\label{prop.classical.solution}
 Assume that $D(P^{n}) \ni f \to \pa^{\alpha} f \in
 \maC (\overline{\Omega})$ is continuous for all $|\alpha| \le m$.
 Assume in addition that $P$ generates an analytic semi-group on $X$
 and that $F = 0$. Then, for all $h \in X$,  $u(t) := e^{tP}h$ is a classical
solution on $(0, \infty)$ of
 the IVP \eqref{eq.evolution.form}.
\end{proposition}

\subsection{Function spaces} \label{sec.spaces}
We shall consider various weighted Sobolev spaces as follows.  Let
$\Omega \subset \RR^d$ be an open subset, as in the previous
subsection, and let $w \in L^1_{loc}(\Omega)$ satisfy $w \ge 0$.  If
$X$ is any Banach space of functions on $\Omega$ with norm $\| \cdot
\|_{X}$, we define
\begin{equation}\label{eq.def.wX}
 wX \ede \{\, w\xi, \ \xi \in X\, \} \,,
\end{equation}
with the norm $\| w \xi\|_{wX} := \|\xi\|_{X}$. Thus, if $p < \infty$,
if $X = L^p(\Omega, d\mu)$, and if $w > 0$ almost everywhere with
respect to $\mu$, $\mu \ge 0$, then $wX = L^p(\Omega, w^{-1/p} d\mu)$.
Of course, for any linear operator $T$ we have
\begin{equation}\label{eq.conjugation}
 T : wX \to wX \ \mbox{ is bounded if, and only if } \ w^{-1} T w : X
 \to X \ \mbox{ is bounded.}
\end{equation}
In fact, these two operators are unitarily equivalent.

In what follows, we choose weights of the form \ $w := e^{\lambda
  \weight}$, where $\langle \, \cdot \, \rangle$ denotes the Japanese bracket:
\begin{equation*}
 \weight \ede \sqrt{1 + x^{2}},
\end{equation*}
and $\lambda \in \RR$ is a parameter. The weight $w$ will be viewed as
acting on functions of $x \in \RR$ or of $(\sigma, x) \in I \times
\RR$, in the later case the weight being independent of $\sigma$.  For
simplicity, we shall usually write
\begin{multline}\label{eq.def.L2}
 H^m_{\lambda}(\RR) \ede  e^{\lambda \weight} H^m(\RR) \, = \,
 \{\, f : \RR \to \CC, \,  e^{-\lambda \weight}f \in H^m(\RR)  \, \}\\
 \, = \, \{\, f : \RR \to \CC, \  e^{-\lambda \weight} \pa^i f \in L^2(\RR),\  \,
 i \le m \, \}\, ,
\end{multline}
where the last equality is valid due to the fact that the weight $w(x)
= e^{\lambda \weight}$ has the property that $w^{-1} \pa^i w$ forms a
bounded family as operators on $H^m(\RR)$ (by writing $f = w g$, with
$g \in H^{m}(\RR)$). We also let $L_\lambda^2 = H_\lambda^0$.  We
recall that this choice of the weight function $w$ is justified by the
specific form of the initial data $h(x) := |e^{x} - K|_{+}$ for the
Cauchy problem for the $\lambda$SABR model \eqref{eq.def.lSABR}.

Let $I$ be a closed interval in $\RR$.
We consider, similarly, the spaces
\begin{multline}\label{eq.mixed.Sobolev}
 H^{i,j}_{\lambda}(I \times \RR)) \ede w H^i(I; H^{j}(\RR)) \, = \, \{
 \, u , \ \pa_\sigma^{\alpha} \pa_x^{\beta} u \in L^2_\lambda(I \times
 \RR), \ \alpha \le i, \beta \le j \, \}\, \\
 = \, \{ \, u, \ \pa_\sigma^{\alpha} \pa_x^{\beta} ( e^{- \lambda
   \weight} \, u ) \in L^2(I \times \RR), \ \alpha \le i, \beta \le j
 \, \}
 = \, H^i(I; H^{j}_{\lambda}(\RR)) \,.
\end{multline}

\subsection{Operators with totally bounded coefficients}

Let $\Omega = \RR$ or $\Omega = I \times \RR$, with $I\subset \RR$ an
interval.  We shall often use the following class of functions.

\begin{definition}
 A function $f : \Omega \to \CC$ is {\em totally bounded} if it is
 smooth and bounded and all its derivatives are also bounded.
\end{definition}

We have the following simple lemma.

\begin{lemma} \label{lemma.cont.bg}
Let $P$ be an order $m$ differential operator on $\Omega$ with totally
bounded coefficients. Then $P$ defines continuous a map
$H^{s}_{\lambda}(\Omega) \to H^{s-m}_{\lambda}(\Omega)$, for every $s
\ge m$.
\end{lemma}

\begin{proof}
 The proof is a direct calculation.
\end{proof}

\begin{lemma}\label{lemma.conjugation}
 Let $P := \sum_{|\alpha|\le m} a_{\alpha} \pa^{\alpha}$ be an order
 $m$ differential operator on $\Omega$ with totally bounded
 coefficients. If $w(\sigma, x) = e^{\lambda \weight}$, as before,
 then $w^{-1} P w$ also has totally bounded coefficients and the same
 terms of order $m$ as $P$.
\end{lemma}

\begin{proof}
 We have $w^{-1} \pa_{\sigma} w = \pa_{\sigma}$ and $w^{-1} \pa_{x} w
 = \pa_{x} + w^{-1} \frac{\pa w}{\pa x} = \pa_x + \psi$, where $\psi
 := w^{-1} \frac{\pa w}{\pa x} = \lambda \frac{\pa \weight}{\pa x} =
 \lambda \weight'$. Since $\weight'$ is totally bounded, the result
 follows from the identity $(w^{-1} P_1 w)(w^{-1} P_2 w) = w^{-1}
 P_1P_2 w$ for any differentiable operators $P_1$ and $P_2$.
\end{proof}

We formulate the following result in slightly greater generality than
needed for the proof of the existence of the semi-group generated by
$L$, for further possible applications. Let us now recall the
definition of a second order uniformly strongly elliptic differential
operator on $\Omega = \RR$ or $\Omega = I \times \RR$, with real
coefficients, in the form that we will use in this paper.

\begin{definition}\label{def.u.s.e}
Let $P = a_{xx}(\sigma, x) \pa_x^{2} + 2a_{\sigma x}(\sigma, x)
\pa_\sigma \pa_x + a_{\sigma \sigma}(\sigma, x) \pa_{\sigma}^{2} +
b(\sigma, x) \pa_x + c(\sigma, x) \pa_{\sigma} + d(\sigma, x)$ be a
differential operator with real coefficients on $I \times \RR$.  We
say that $P$ is {\em uniformly strongly elliptic} if it has bounded
coefficients and if there exists $\epsilon > 0$ such that $a_{xx} \ge
\epsilon$ and $a_{xx} a_{\sigma \sigma} - a_{\sigma x}^2 \ge
\epsilon$.
\end{definition}

If $\Omega = \RR$, the operator $P$ reduces to $P = a_{xx}(\sigma, x)
\pa_x^{2} + b(\sigma, x) \pa_x + d(\sigma, x)$ and we have that $P$ is
uniformly strongly elliptic if it has bounded coefficients and there
exists $\epsilon > 0$ such that $a_{xx} \ge \epsilon$.

We have the following standard regularity results. We continue to
assume that $\Omega = I \times \RR$ or $\Omega = \RR$.

\begin{theorem}\label{thm.reg}
Let $P$ be second order, uniformly strongly elliptic differential
operator with totally bounded coefficients on $\Omega$. Assume $u \in
H^{1}_{\lambda}(\Omega)$ is such that $Pu \in
H_{\lambda}^{m-1}(\Omega)$.
If $\Omega = I \times \RR$, we also assume
that $u$ vanishes at the endpoints of $I$.
Then $u \in
H_{\lambda}^{m+1}(\Omega)$.  Moreover, there exists $C > 0$,
independent of $u$, such that $\|u\|_{H_{\lambda}^{m+1}(\Omega)} \le C
\big ( \|P u\|_{H_{\lambda}^{m-1}(\Omega)} +
\|u\|_{H_{\lambda}^{1}(\Omega)} \big )$.
\end{theorem}

A proof of this result can be obtained by first reducing to the case
$\lambda = 0$, that is, $w=1$, using Lemma \ref{lemma.conjugation} and
then either by using a dyadic partition of unity or by using divided
differences (this approach is sometimes called Nirenberg's trick after
\cite{ADN59}), which is facilitated in this case since the boundary is
straight (see, for instance, \cite{LionsMagenes1}).
We obtain the following consequence.

\begin{corollary}\label{cor.reg}
Let $P$ be second order, uniformly strongly elliptic differential
operator with totally bounded coefficients on $\RR$.  Then
$\|u\|_{L^2_{\lambda}} + \|P^k u\|_{L^2_{\lambda}}$ defines an
equivalent norm on $H^{2k}_{\lambda}(\RR)$.
\end{corollary}

The above two results hold in the more general framework of manifolds
with bounded geometry. See, for example, \cite{MazzucatoNistor1} and
the references therein. See also \cite{amannFunctSpaces, amannAnis, sobolev, Arendt,
Disconzi, nadine, kordyukovLp1} for more
recent results on PDEs on manifolds with bounded geometry.

\section{The semi-group generated by $L$ and $B$\label{sec3}}

In this section, we show that $L$ and $B$ generate analytic
semi-groups using the Lumer--Phillips theorem and the results of the
previous section. Our approach is standard and well known for the case
of operators on standard Sobolev spaces. The analysis on exponentially
weighted spaces is less developed. For the reader's sake, we work in
detail the slightly more complicated case of the operator $L$ and only
sketch the proofs of the results for $B$.

\subsection{The differential operator $L$}
Our next goal is to show that the operator $L$ is quasi-dissipative on
weighted Sobolev spaces.  The space
\begin{equation}\label{def.maK}
 \maK_0 \ede H^{2}_{\lambda}(I \times \RR) \cap \{u = 0 \mbox{ on }
 \pa I \times \RR \}
\end{equation}
will be the common domain of several operators, so it will play an
important role in what follows. We formulate the following result in
slightly greater generality than needed for the proof of the existence
of the semi-group generated by $L$, for further possible applications.

\begin{definition} \label{def.calP}
Let $\maP$ denote the set of second order differential operators $T =
a_{xx}(\sigma, x) \pa_x^{2} + 2a_{\sigma x}(\sigma, x) \pa_\sigma
\pa_x + a_{\sigma \sigma}(\sigma, x) \pa_{\sigma}^{2} + b(\sigma, x)
\pa_x + c(\sigma, x) \pa_{\sigma} + d(\sigma, x)$ with totally
bounded, real coefficients on $I \times \RR$ and satisfying
\begin{equation*}
 a_{xx}, \ a_{\sigma\sigma}, \ a_{xx}a_{\sigma\sigma} - a_{\sigma x}^2
 \ge 0 \,.
\end{equation*}
\end{definition}

For $T$ as in this definition, we shall denote
\begin{equation}\label{eq.def.matrix}
M_T \ede \left [
 \begin{array}{cc}
  a_{xx} & a_{\sigma x}\\
  a_{\sigma x} & a_{\sigma x}
 \end{array}
 \right ]
\end{equation}
the matrix determined by its highest order coefficients (the
principal symbol) of~$T$.

\begin{proposition}\label{prop.dissipative}
If $w(\sigma, x) = e^{\lambda \weight}$ and $T \in \maP$, then
$w^{-1}Tw \in \maP$.  Let $M_T$ be as in Equation \eqref{eq.def.matrix},
then there exists $C>0$ such that
\begin{equation*}
  (Tu, u)_{L^2_{\lambda}(I \times \RR)} \, \le \, - \int_{I \times
    \RR} (M_T \nabla u, \nabla u) e^{-2\lambda \weight}\, d\sigma dx \,
  + \, C \|u\|^{2}_{L^2_{\lambda}(I \times \RR)}\, , \ \quad u \in
  \maK_0\,,
\end{equation*}
and hence, $T$ with domain $\maK_0 := H^{2}_{\lambda}(I \times \RR)
\cap \{u = 0 \mbox{ on } \pa I \times \RR \}$ is quasi dissipative on
$L^2_{\lambda}(I \times \RR)$.
\end{proposition}

\begin{proof}
The fact that $w^{-1}Tw$ is of the same form as $T$ follows from Lemma
\ref{lemma.conjugation}.  In view of Equation \eqref{eq.conjugation},
we can assume that $\lambda = 0$, that is, $w := e^{\lambda \weight} =
1$.  The rest of the proof is then a well-known direct calculation,
which we include for the benefit of the reader. Since we work with
Hilbert spaces, we can take $f^*(\xi) = (\xi, f)$ in the condition
defining the quasi dissipativity. We first notice that, by changing
$b$, $c$, and $d$, we can assume that $Tu = \pa_x (a_{xx} \pa_x u) +
\pa_\sigma (a_{\sigma x} \pa_x u) + \pa_x (a_{\sigma x} \pa_\sigma u)
+ \pa_\sigma (a_{\sigma \sigma} \pa_\sigma u) + b\pa_x u + c
\pa_{\sigma}u + d u$. Then, we perform a standard energy estimate,
in which the integration by parts is justified by the fact that
$u \in \maK_0$:
\begin{multline}\label{eq.coeff.c}
 2 \Re \big( c \pa_\sigma u, u) \ede 2 \Re \int_{I \times \RR} c
 (\pa_\sigma u) \overline{u} \, d\sigma dx \ = \ \int_{I \times \RR} c
 (\pa_\sigma u) \overline{u} \, d\sigma dx \\
 \, + \, \int_{I \times \RR} c (\pa_\sigma \overline{u}) u \, d\sigma
 dx
 = \ \int_{I \times \RR} \pa_\sigma (c |u|^2) \, d\sigma dx \, - \,
 \int_{I \times \RR} (\pa_\sigma c) |u|^2 \, d\sigma dx
 \\ = \ \int_{\RR} \big (c(\beta, x) |u(\beta, x)|^2 - c(\alpha, x)
 |u(\alpha, x)|^2 \big ) \, dx - \int_{I \times \RR} (\pa_\sigma c)
 |u|^2 \, d\sigma dx
 \\ = \ - \int_{I \times \RR} (\pa_\sigma c) |u|^2 \, d\sigma dx\,,
\end{multline}
using that $c$ is real valued and the fact that  $u \in \maK_0$.
 Similarly, since $b$ is also real valued,
\begin{multline}\label{eq.coeff.b}
 2 \Re \big( b \pa_x u, u) \ede 2 \Re \int_{I \times \RR} b (\pa_x u)
 \overline{u} \, d\sigma dx \ = \ \int_{I \times \RR} b (\pa_x u)
 \overline{u} \, d\sigma dx \\
 \, + \, \int_{I \times \RR} b (\pa_x \overline{u}) u \, d\sigma dx
 = \ \int_{I \times \RR} \pa_x (b |u|^2) \, d\sigma dx \, - \, \int_{I
   \times \RR} (\pa_x b) |u|^2 \, d\sigma dx
 \\ = \ - \int_{I \times \RR} (\pa_x b) |u|^2 \, d\sigma dx \,.
\end{multline}

Next, we consider the quadratic terms.
By assumption, all the eigenvalues of the matrix $M_T$ of
Equation \eqref{eq.def.matrix}
are non-negative. Let $\delta$ the smallest of the eigenvalues of $M_T$.
We obtain:
\begin{multline}\label{eq.coeff.a}
 - \big( \pa_x (a_{xx} \pa_x u) + \pa_\sigma (a_{\sigma x} \pa_x u) +
 \pa_x (a_{\sigma x} \pa_\sigma u) + \pa_\sigma (a_{\sigma \sigma}
 \pa_\sigma u), u \big ) \\
 = \, \int_{I \times \RR} \Big ( a_{xx} (\pa_x u) \pa_x \overline{u} +
 a_{\sigma x} (\pa_\sigma u) \pa_x \overline{u} + a_{\sigma x} (\pa_x
 u) \pa_\sigma \overline{u} + a_{\sigma \sigma} (\pa_\sigma u)
 \pa_\sigma \overline{u} \Big )\, d\sigma dx \\
 \seq \int_{I \times \RR} (M_T \nabla u, \nabla u) \, d\sigma dx \ \ge
 \ \delta \int_{I \times \RR} |\nabla u(\sigma, x) |^{2} \, d\sigma
 dx \,,
\end{multline}
and the last term is positive since the quadratic form defined by
$a_{xx}$, $a_{\sigma x}$, and $a_{\sigma x}$ is positive, by
assumption.

Combining Equations \eqref{eq.coeff.c}, \eqref{eq.coeff.b}, and
\eqref{eq.coeff.a}, we obtain
\begin{multline}\label{eq.Garding}
 2 \Re \big( T u, u) \, \le \, - \int_{I \times \RR} (M \nabla u,
 \nabla u) \, d\sigma dx \, d\sigma dx
 - \int_{I \times \RR} \big( \pa_x b + \pa_\sigma c + d \big ) |u|^2
 \, d\sigma dx\\
\le \, - \int_{I \times \RR} (M_T \nabla u, \nabla u) \, d\sigma dx \, +
\, C \|u\|^{2} \, \le \, - \delta \int_{I \times \RR} \|\nabla
u(\sigma, x) \|^{2} \, d\sigma dx \, + \, C \|u\|^{2} \,,
\end{multline}
where $C = \|\pa_x b + \pa_\sigma c + d\|_{\infty}$.  The fact that
$T$ is quasi-dissipative follows since $\delta \ge 0$.
\end{proof}

Let us note also, for further reference, the following consequences of
the calculation in the above proof.

\begin{corollary}\label{cor.cont}
 Let $T$ be as in Proposition \ref{prop.dissipative}. Then there
 exists a constant $C> 0$ such that $|(Tu, u)| \le
 C\|u\|_{H^{1}_{\lambda}(I \times \RR)}$
\end{corollary}

\begin{proof}
 This is a simple calculation, very similar to those in the proof of
 Proposition \ref{prop.dissipative}. In particular, we can assume
 $\lambda = 0$.  The main difference is with Equation
 \eqref{eq.for.a}, which is replaced by
 \begin{multline*} 
 0 \le - \big( \pa_x (a_{xx} \pa_x u) + \pa_\sigma (a_{\sigma x} \pa_x
 u) + \pa_x (a_{\sigma x} \pa_\sigma u) + \pa_\sigma (a_{\sigma
   \sigma} \pa_\sigma u), u \big ) \\
 = \, \int_{I \times \RR} \Big ( a_{xx} (\pa_x u) \pa_x \overline{u} +
 a_{\sigma x} (\pa_\sigma u) \pa_x \overline{u} + a_{\sigma x} (\pa_x
 u) \pa_\sigma \overline{u} + a_{\sigma \sigma} (\pa_\sigma u)
 \pa_\sigma \overline{u} \Big )\, d\sigma dx \\
 \le \ \mu \int_{I \times \RR} \|\nabla u(\sigma, x) \|^{2} \, d\sigma
 dx \, \le \, \mu \|u\|_{H^1(I \times \RR)}^{2}\,,
 \end{multline*}
 where $\mu$ is the largest of the eigenvalues of the matrix $M_T$ of
 Equation \eqref{eq.def.matrix}.
\end{proof}

Garding's inequality also holds in our setting. We have the opposite
sign to the one that is typically used, as we work with
negative-definite operators.

\begin{corollary}\label{cor.Garding}
Let $T$ be as in the statement of Proposition \ref{prop.dissipative}.
Assume also that there exists $\epsilon > 0$ such that
$a_{xx}a_{\sigma\sigma} - a_{\sigma x}^2 \ge \epsilon$.
Then there
exist $C_1 > 0$ and $C_2$ such that
\begin{equation*}
 \Re (Tu, u) \, \le \, - C_1\|u\|_{H^{1}_{\lambda}(I \times \RR)}^{2}
 + C_2 \|u\|_{L^{2}_{\lambda}(I \times \RR)}^{2}\,.
\end{equation*}
Also, if $u \in H^1_{\lambda}(I \times \RR) \cap \{u\vert_{\pa I
  \times \RR} = 0\}$ satisfies $Tu \in L^2_{\lambda}(I \times \RR)$,
then $u \in H^{2}_{\lambda}(I \times \RR)$.  Consequently, $T - \mu_0
: \maK_0 \to L^2_{\lambda}(I \times \RR)$ is invertible for $\mu_0 >
C_2$.
\end{corollary}

\begin{proof}
Garding's inequality is an immediate consequence of Equation
\eqref{eq.Garding}. The rest is a consequence of this inequality, and
we only outline the main steps in the proof.

First of all, by Lemma \ref{lemma.conjugation}, we can assume that
$\lambda = 0$.  Let $\Omega := I \times \RR$. Garding's inequality
allows us to invoke the Lax-Milgram Lemma, which gives that $T -
\mu_0 : H^1(I \times \RR) \cap \{u\vert_{\pa \Omega} = 0\} \to
H^{-1}(I \times \RR)$ is invertible for $\mu_0 > C_2$ (that is, $T -
\mu_0$ is a continuous bijection with continuous inverse).

By replacing $T$ with $T - \mu_0$, if necessary, we can assume that $T
: H^1(I \times \RR) \cap \{u\vert_{\pa \Omega} = 0\} \to H^{-1}(I
\times \RR)$ is invertible. The assumptions on our coefficients (that
they are bounded and that $a_{xx} \ge 0$, $a_{\sigma\sigma}\ge 0$, and
$a_{xx}a_{\sigma\sigma} - a_{\sigma x}^2 \ge \epsilon > 0$) imply that
$T$ is uniformly strongly elliptic (see Definition
\ref{def.u.s.e}). Therefore, it satisfies elliptic regularity (Theorem
\ref{thm.reg}). In particular, if $u \in H^1(I \times \RR) \cap
\{u\vert_{\pa \Omega} = 0\}$ is such that $Tu \in L^2(I \times \RR)$,
then $u \in H^2(I \times \RR)$ and hence, by taking into account that
$u$ vanishes at the boundary, $u \in \maK_0$. We finally obtain that
\begin{equation*}
 T : \maK_0 \ede H^{2}_{\lambda}(I \times \RR) \cap \{u = 0 \mbox{ on
 } \pa \Omega = \pa I \times \RR \} \, \to \, L^2(I \times \RR)
\end{equation*}
is both injective and surjective, and hence it is invertible.  (The
continuity of the inverse follows either from abstract principles,
namely from the Open Mapping Theorem, or, constructively, from Theorem
\ref{thm.reg}.)
\end{proof}

We obtain as a consequence the following theorem.

\begin{theorem}\label{thm.gen.L}
 Let $T$ be as in the statement of Corollary \ref{cor.Garding}. Then
 $T$ generates an analytic semi-group $e^{tL}$ on $L^2_{\lambda}(I
 \times \RR)$.  In particular, if $I = (\alpha, \beta)$ is a bounded
 interval with $0 < \alpha \le \beta < \infty$, then $L$ as given in
 \eqref{eq.def.lSABR} satisfies the hypothesis of Corollary
 \ref{cor.Garding}, and hence it generates an analytic semi-group on
 $L^2_{\lambda}(I \times \RR)$.
\end{theorem}

\begin{proof}
 Corollaries \ref{cor.cont} and \ref{cor.Garding} show that the $T$
 satisfies the assumptions of Lemma \ref{lemma.n.r} (that is, $T$ is
 continuous and satisfies a Garding-type inequality). Since $T -
 \mu_0$ is invertible for $\mu_0$ large, again by Corollary
 \ref{cor.Garding}, we are in position to use Corollary
 \ref{cor.analytic} to conclude that $T$ generates an analytic
 semi-group.  If $I$ is bounded, then $L$ has totally bounded
 coefficients. Since $\alpha > 0$, $L$ is also uniformly strongly
 elliptic, and the first part of the result applies.
\end{proof}

\begin{corollary}
 Let $T$ be as in Theorem \ref{thm.gen.L} and $h \in L^2_{\lambda}(I
 \times \RR)$, for some $\lambda \in \RR$. Then $u(t) := e^{tT}h$ is a
 strong solution of $\pa_t u - T u = 0$, $u(0) = h$. It is also a classical
 solution on $(0,\tau]$, for all $\tau > 0$. Moreover, $u(t)$ does not depend on $\lambda$.
\end{corollary}

\begin{proof}
 We have that $T$ generates an analytic semi-group $S(t) = e^{tT}$.
 Moreover, elliptic regularity gives $D(T^{k}) \subset H^{2k}(I \times
 \RR)$, for all $k \in \ZZ_{+}$. The Sobolev embedding theorem then
 gives us that the assumptions of Lemma \ref{lemma.classical.s} and
 Proposition \ref{prop.classical.solution} are satisfied. This proves
 the first part of the result.

 The independence of $u$ on $\lambda$ follows from the fact that the
 map $L^2_{\lambda'}(I \times \RR) \to L^2_{\lambda''}(I \times \RR)$ is
 injective and continuous for all $\lambda' < \lambda''$ and from the
 uniqueness of strong solutions.
\end{proof}

\begin{remark}
 The assumption that $I$ be a bounded interval in the second half
 of Theorem \ref{thm.gen.L}, is essential for our
 method to apply. Our method does not apply, for instance, if $I = (0,
 \infty)$. The problem lies in the fact that, at $\sigma = 0$, we lose
 uniform ellipticity and, at $\sigma = \infty$, the coefficient
 $\theta - \sigma$ becomes unbounded. However, if $\kappa = 0$, we do
 obtain that $L$ generates an analytic semi-group using the results in
 \cite{MazzucatoNistor1}. The degeneracy at $\sigma=0$ and
 $\sigma=\infty$ could be addressed by introducing appropriate weights
 in $\sigma$. For the applications of interest in this work, it is
 enough to consider $\sigma$ in a bounded interval, bounded away from
 zero.
\end{remark}

\subsection{The differential operator $B$}

We now consider the operator $B := \pa_x^2 - \pa_x$ (recall Equation
\ref{def.operators}). The fact that $B$ generates an analytic semigroup is
classical. However, since we work with exponentially weighted spaces,
we state needed results for clarity and completeness. We start by
collecting all the needed technical facts about $B$ in the following
proposition, which we state in more generality than actually needed.
It an be seen as a special case of the analysis of the
operator $L$ (see also Lemma \ref{lemma.conjugation}).

\begin{proposition}\label{prop.dissipative2}
 Let $T = a \pa_x^2 + b \pa_x + c$ be a uniformly strongly elliptic
 operator with totally bounded coefficients. Then:
 \begin{enumerate}[(i)] 
  \item $wTw^{-1}$ is also strongly elliptic with totally bounded
    coefficients.
  \item There is $C_3 > C_1 > 0$ and $C_2 \in \RR$ such that, for any
    $u \in H^2_{\lambda}(\RR)$,
  \begin{equation*}
    \Re (Tu, u) \, \le \, - C_1\|u\|_{H^{1}_{\lambda}(\RR)}^{2} + C_2
    \|u\|_{L^{2}_{\lambda}(\RR)}^{2}\ \ \mbox{ and } \ \ |(Tu, u)| \,
    \le \, C_3\|u\|_{H^{1}_{\lambda}(\RR)}^{2}\,.
  \end{equation*}
  \item $T- \mu_0 : H^2_{\lambda}(\RR) \to L^{2}_{\lambda}(\RR)$ is
    invertible for $\mu_0 > C_2$.
 \end{enumerate}
\end{proposition}

Using the same argument as for Theorem \ref{thm.gen.L},
we obtain the following result.

\begin{theorem}\label{thm.gen.B}
 Let $T$ be as in the statement of Proposition
 \ref{prop.dissipative2}. Then $T$ generates an analytic semi-group on
 $L^2_{\lambda}(\RR)$. In particular, $B$ generates an analytic
 semi-group on $L^2_{\lambda}(\RR)$.
\end{theorem}

Since $D(B^{k}) = H^{2k}_{\lambda}(\RR)$, we also obtain the following.

\begin{corollary}\label{cor.HkB}
 The operator $B$ generates an analytic semi-group on $H^{j}_{\lambda}(\RR)$,
 for all $j$.
\end{corollary}

Again using the same argument as in the previous subsection, we have
also the following result.

\begin{corollary}
 Let $T$ be as in Proposition \ref{prop.dissipative} and $h \in
 L^2_{\lambda}(\RR)$, for some $\lambda \in \RR$. Then $u(t) :=
 e^{tT}h$ is a strong solution of $\pa_t u - T u = 0$, $u(0) = h$. Moreover, it
is
 a classical solution on any interval $(0,\tau]$, $\tau>0$, and
$u(t)$ does not depend on $\lambda$.
\end{corollary}

\begin{remark}
 In view of the independence of $\lambda$, we obtain that the
 semi-group $e^{tB}$ is given by the following explicit formula
\begin{equation}\label{eq.exptb}
  e^{tB}h(x) \seq \frac1{\sqrt{4\pi t}}\int e^{-\frac{|x-y-t|^2}{4t}}
  h(y) \, dy \,.
\end{equation}
 In particular, if $\lambda = 0$, the semi-group generated by $B$
 consists of contractions.
\end{remark}

We will need on several occasions the following well-known
lemma. In particular, we will need it to treat families of operators.
(We note here that this lemma will be generalized to deal with
differentiability in the strong sense in Lemma~\ref{lemma.diff}.)

\begin{lemma}\label{lemma.cont}
Let $\xi \in \maC([0, 1]; X)$ and $[0, 1] \ni t \to V(t) \in \maL(X)$
be strongly continuous.  Then the map $[0, 1] \ni t \to V(t) \xi(t)
\in X$ is continuous.
\end{lemma}

In order to apply our results on analytic semi-groups to Equation
\eqref{eq.evolution.form}, we will need to consider families of
operators. In particular, we will show that the operator $P =
\frac{\sigma^{2}}{2} B$, acting on functions of $\sigma$ and $x$,
that appears in the $\lambda$SABR PDE, also generates an analytic
semigroup. If $p : I \to [0, \infty)$
is bounded and continuous,
then we shall write $pB$ for the operator $(pB v)(\sigma) = p(\sigma) B v(\sigma)
\in L^2_{\lambda}(\RR)$ and $e^{pB}$ for the operator $(e^{pB} v)(\sigma) = e^{p(\sigma)B} v(\sigma)
\in L^2_{\lambda}(\RR)$, where $v : I \to H^2_{\lambda}(\RR)$. We thus regard both $pB$ and
$e^{pB}$ as a family of operators parameterized by $\sigma \in I$ and acting on
$L^2_{\lambda}(\RR)$-valued functions defined on $I$.

\begin{proposition}\label{prop.fam.B}
Let $T$ be a differential operator as in Proposition
\ref{prop.dissipative2}.  Let $I \subset \RR$ be an interval and $p :
I \to [0, \infty)$ be a {\em bounded} continuous function. Then $e^{tpT}$,
defined by the
  formula $(e^{tpT}h)(\sigma) := e^{t\,p(\sigma)\,T}h(\sigma) \in L^2_{\lambda}(\RR)$,
  $\sigma \in I$, defines a $c_0$
  semi-group on $L^{2}_{\lambda}(I \times \RR)$ with generator
  $pT$.
\end{proposition}

\begin{proof}
 Since $T$ generates a $c_0$ semi-group, $e^{t\,p(\sigma)\, T}h(s)$ depends
 continuously on $\sigma \in I$ whenever $h \in L^{2}_{\lambda}(I \times
 \RR)$ is continuous in $\sigma$.  Since $\|e^{t T}\|$ is uniformly bounded
 for $t$ in a bounded interval, we obtain that the family of operators
 $e^{t\,p(\sigma)\, T}$ thus defines a bounded operator on
 $L^{2}_{\lambda}(I \times \RR)$.
\end{proof}

We shall need the following extension of Lemma \ref{lemma.cont}.

\begin{lemma}\label{lemma.diff}
 Let $J := (0, 1)$ and assume that $\xi \in \maC^1(J; X)$, that $T$ is
 the generator of $c_0$ semi-group $V(t)$ on $X$, and that one of the
 following two conditions is satisfied:
 \begin{enumerate}[(i)] 
  \item $\xi(t) \in D(T)$ and the map $J \ni t \to T\xi(t) \in X$ is
    continuous;
  \item the semi-group $V(t)$ generated by $T$ is an analytic
    semi-group.
 \end{enumerate}
 Then $V(t) \xi(t) \in \maC^{1}(J; X)$ with differential $TV(t)\xi(t)
 + V(t) \xi'(t)$.
\end{lemma}

Let $\maK_1 := H^{2}_\lambda(I \times \RR)$, as in Corollary
\ref{cor.sgA} in the previous subsection.

\begin{corollary}\label{cor.C1}
  Let $f : [\alpha , \beta ] = \overline{I} \to [\epsilon, \infty)$,
  $\epsilon > 0$. Assume that $f$, $f'$, and $f''$ are (defined and)
    continuous. Then $e^{fB}$ maps $\maK_1$ to itself. Moreover,
    $e^{tfB}$ defines a $c_0$-semigroup on $\maK_1$, generated by $fB$ as an
operator with domain
\begin{equation*}
 \{\xi \in \maK_1,\, B\xi \in \maK_1 \}  \,
 \supset \, H^{2,4}_{\lambda}(I \times \RR)
 \seq H^2(I; H^{4}_{\lambda}(\RR)) \,.
\end{equation*}
\end{corollary}

\begin{proof}
 The first part is an immediate consequence of Lemma \ref{lemma.diff}(ii)
 and of Remark \ref{rem.a.est}. The second part follows using also
 Corollary \ref{cor.HkB}.
\end{proof}

\section{The semi-group generated by $L_0$\label{sec4}}

In this section, we discuss the derivation of an explicit formula for
the distributional kernel of the operator $e^{t L_0}$ using Lie
algebra techniques. Besides being of independent interest, in this
work we utilize the explicit formula for $e^{t L_0}$ to approximate
$e^{t L}$, for which no closed form are available. This is achieved by
means of a perturbative expansion in the parameter $\nu$, the
so-called volvol or volatility of the volatility.
We recall that $L_0 = A +\frac{\sigma^2}{2}\, B$ and $L = L_0 + \nu
L_1 + \nu^2 L_2$, with $L_i$ independent of $\nu$ (see Equations
\eqref{eq.decomp.L} and \eqref{def.operators}).

There is an added difficulty in our problem,
namely, the fact that $L_0$ is not strongly elliptic, and
$\pa_t-L_0$ is not hypoelliptic in the sense of H\"ormander
\cite{Hormanderbook} (although $L_0$ is). As a matter of fact, this
expansion is only valid under additional regularity assumptions on the
initial data $h$, which will be discussed in Section \ref{sec5}.

The explicit formula for $e^{t L-0}$ is derived from the corresponding
formulas for $e^{tA}$ and $e^{\frac{t \sigma^2}{2} B}$, where the
later is defined using Proposition \ref{prop.fam.B}. Our approach can
be viewed as akin to an operator splitting argument, where the
hyperbolic and parabolic parts of $L_0$ are treated separately,
although we do not explicitly resort to any splitting in the PDE
itself.

{\em We thus assume that $I = (\alpha, \beta)$ satisfies $0 < \alpha <
  \theta < \beta < \infty$, as in Proposition \ref{prop.fam.B}.  We
  will make the further assumption that $\kappa > 0$.}

This last assumption implies that the characteristics of the operator
$A$ are incoming at $\sigma = \alpha$ and $\sigma =\beta$, as
long as $\alpha < \theta < \beta$ and $\kappa > 0$.  Therefore, no
boundary conditions need to be imposed at $\sigma=\alpha$ and $\sigma
=\beta$ (cf. the seminal
work of Feller \cite{Feller51,Feller52}). The case $\kappa < 0$ is
similar provided one imposes suitable boundary conditions. However,
this case will not be needed for our purposes.

We now study $e^{tA}$ and its properties. These will be used in
deriving an explicit formula for $e^{tL_0}$.

\subsection{The transport equation}
Let $I = (\alpha, \beta) \subset \RR$ and $A := \kappa (\theta -
\sigma) \pa_{\sigma}$, as before. We consider the transport equation
\begin{equation}\label{eq.for.a}
 \pa_t v - A v \, = \, 0 \,,
\end{equation}
where $v$ depends on $\sigma$ and, possibly, on some parameters.
 Let
\begin{equation}\label{eq.def.delta}
 \delta_t(\sigma) \ede \theta(1-e^{-\kappa t}) + \sigma e^{-\kappa
   t}\,,
\end{equation}
which satisfies $\delta_t(I) \subset I$, by our assumptions on $I$,
and $\delta_{t}\circ\delta_s = \delta_{t+s}$.

Most of the results listed below are classical, at least for
$\lambda=0$. We state and prove results in the form needed for our
purposes for clarity and completeness.

\begin{lemma}\label{lemma.hyperbolic}
Let $h \in L_{loc}^1(I)$, and let $v$ be given by the formula
\begin{equation}\label{eq.expta}
 v(t, \sigma) \ := \ h(\delta_t(\sigma))\,.
\end{equation}
Then $v$ is a weak solution of \eqref{eq.for.a} on $[0, \infty) \times
  \RR$ with $v(0) = h$ (i.e. $v(0, \sigma) = h(\sigma)$).  If $h \in
  \maC^{1}(I)$, then $v$ is also a classical solution this equation.
\end{lemma}

\begin{proof}
 The proof that $v$ is a classical solution if $h \in \maC^{1}(I)$ is
 by a direct calculation. To prove that $v$ is a weak solution in
 general, we can consider the change of coordinates $(t, \sigma) = (t,
 \delta_{-t}(s))$ and then perform an integration by parts in $s$,
 using also Fubini's theorem.
\end{proof}

In what follows, we consider $A$ as operator acting of functions of
$\sigma$ with values in a Hilbert space $\maH$. For the application at
hand, $\maH$ will be an exponentially weighted Sobolev space.
The following proposition justified.

\begin{proposition}\label{prop.sgA} Let $\maH$ be a Hilbert space.
 Let $T(t)h = v(t)$, where $v$ is as in Lemma \ref{lemma.hyperbolic}
 and $h \in L^2(I; \maH)$. Then $\|T(t)h\| \le e^{\kappa t/2} \|h\|$,
 where the norm is the one on $L^2(I; \maH)$. Moreover, $T(t)$ is a
 $c_0$ semi-group whose generator coincides with $A$ on $\maC^1(I;
 \maH)$.
\end{proposition}

\begin{proof} The relation $\|T(t)h\| \le e^{\kappa t/2} \|h\|$ follows by
a change of variables (note also that, for $I = \RR$, we have
equality).  The identity $T(t_1)T(t_2)h=T(t_1+t_2)h$ follows from
$\delta_{t_2}(\delta_{t_1}(\sigma)) = \delta_{t_1+t_2}(\sigma)$. If $h
\in \maC^1(I; \maH)$, we obtain from the definition that $t^{-1}(T(t)
h - h) \to A h$. Since $\|T(t)\|$ is uniformly bounded for $t \le 1$,
this gives that $T(t)h \to h$ as $t \to 0$ for all $h$. This completes
the proof.
\end{proof}

Below, we shall write $T(t) = e^{tA}$, a notation that is justified by
Corollary \ref{cor.sgA}.

\begin{corollary}\label{cor.sgA.new}
 Using the notation of Proposition \ref{prop.sgA}, we have that $v$ is
 a strong solution of Equation \eqref{eq.for.a} for $h \in \maC^{1}(I;
 L^2_{\lambda}(\RR))$.  If $h \in \maC^{1}(I; H^1_{\lambda}(\RR))$, it
 is also a classical solution.
\end{corollary}

\begin{proof}
This follows from Lemma \ref{lemma.hyperbolic}, Proposition
\ref{prop.sgA} and the definitions of strong and classical solutions.
\end{proof}

We obtain the following consequences for classical solutions of
Equation \eqref{eq.for.a}.

\begin{corollary}\label{cor.sgA}
Let $\maK_1 \ede H^{2}_\lambda(I \times \RR)$.  Then $\maK_1 \subset
\maC^1(I; L^2_{\lambda}(\RR))$.  Using the notation of Proposition
\ref{prop.sgA}, we have that $v$ is a strong solution of Equation
\eqref{eq.for.a} for $h \in \maK_1$.  Moreover, $e^{tA}(\maK_1)
\subset \maK_1$, $e^{tA}$ defines a $c_0$ semi-group on $\maK_1$, and
hence $v(t) \in \maK_1$.
\end{corollary}

\begin{proof}
The inclusion $\maK_1 \subset \maC^1(I; L^2_{\lambda}(\RR))$ is a
consequence of the Sobolev's embedding theorem. The fact that $v$ is a
strong solution follows from Corollary \ref{cor.sgA.new}, and from the
inclusion $\maK_1 \subset \maC^1(I; L^2_{\lambda}(\RR))$.  The
inclusion $e^{tA} (\maK_1) \subset \maK_1$ and the fact that $e^{tA}$
defines a $c_0$ semi-group on $\maK_1$ follow from the explicit
formula for $e^{tA}$.
\end{proof}

\subsection{The generation property for $L_0$}

It seems difficult to apply the Lumer-Philips Theorem directly to a degenerate
operator of the form of $L_0$. We will therefore adopt a different
strategy and directly prove that $e^{tL_0}$ is a semigroup generated
by $L_0$.

Let $\maK_1 := H^{2}_\lambda(I \times \RR)$, as in Corollary
\ref{cor.sgA} in the previous subsection.
Also, we recall the function $ \delta_t(\sigma) \ede \theta(1-e^{-\kappa
  t}) + \sigma e^{-\kappa t}$ introduced above.

\begin{lemma}\label{lemma.cross}
 Let $g : I \to [0, \infty)$ be a continuous function.  Assume that
   either $g$ is bounded or that the parameter $\lambda = 0$ in
     the definition of the weight $w(x) = e^{\lambda \langle x
       \rangle}$. Then $e^{tA} e^{gB} = e^{(g \circ \delta_t) B}
   e^{tA}$.
\end{lemma}

\begin{proof}
 The result follows from $e^{(g \circ \delta_t) B} \xi \circ \delta_t =
 (e^{g B} \xi) \circ \delta_t = e^{tA} e^{gB} \xi.$
\end{proof}

Let us now define
\begin{equation}\label{eq.def.Df}
    \Df_\kappa(t) \, = \, \Df_\kappa (t, \sigma) \ede
    \frac{(\theta-\sigma)^2}{4\kappa}(1-e^{-2\kappa t}) -
    \frac{\theta(\theta-\sigma)}{\kappa}(1-e^{-\kappa t}) + \frac12
    \theta^2t \,.
\end{equation}

\begin{proposition} \label{prop.Df}
The function $\Df(t,\sigma)$ defined in Equation \eqref{eq.def.Df} is
analytic in $(\kappa, t, \sigma) \in \RR^3$ and satisfies
$\Df(0,\sigma) = 0$ and $\Df(t,\sigma) > 0$ for any $t > 0$ and
any $\sigma \in \RR$.
\end{proposition}

\begin{proof}
The function $\Df(t,\sigma)$ is analytic since the singularity at zero
is removable.  We shall regard $\Df(t,\sigma)$ as a second order
polynomial in $\sigma$ with coefficients that are
functions of the parameters $t$ and $\kappa$.  We have that the
leading coefficient $\frac1{4\kappa}(e^{2\kappa t}-1)$ is always
positive as $t>0$, so we only need to show that the discriminant of
$\Df(t, \sigma)$ is non negative. We let $f(t)$ be the discriminant of
$\Df(t, \sigma)$ (regarded as a second-order polynomial in $\sigma$, as
mentioned above), so that
\begin{equation}
  f(t) \, = \, \frac{\theta^2}{2\kappa^2} \big [ (2 +\kappa
    t)e^{-2\kappa t} - 4e^{-\kappa t} + 2-\kappa t \big ] \;.
\end{equation}
We then have:
\begin{equation*}
\begin{gathered}
 f'(t) \, = \, \frac{\theta^2}{2\kappa} \big [ (3-2\kappa t)e^{2\kappa
     t} - 4e^{\kappa t} + 1 \big ] \quad \mbox{and} \\
  f''(t) \, = \, 2\theta^2 \big [ (1-\kappa t)e^{2\kappa t} -
    e^{\kappa t} \big ] \, = \, 2\theta^2 e^{2\kappa t} \big [
    1-\kappa t - e^{-\kappa t} \big ] \, < \, 0 \ \mbox{ for } t \neq
  0 \;.
\end{gathered}
\end{equation*}
It follows that $f'(t)$ is decreasing, and hence $f'(t)<f'(0)=0$ for
$t > 0$. Consequently, $f(t)$ is also decreasing, which gives
$f(t)<f(0)=0$ for positive $t$.
\end{proof}

This lemma allows us to define $e^{\Df(t) B}$ if $I$ is bounded or if
$\lambda = 0$. We let then
\begin{equation}\label{eq.def.S(t)}
 S(t) \ede e^{\Df(t) B} e^{tA}\,.
\end{equation}
Then $S(t)$ is a bounded operator, since it is the composition of bounded
operators.

We will establish that $S(t)$ is a $co$-semigroup generated by $L_0$
by splitting the proof in a few Lemmas for convenience.

\begin{lemma}\label{lemma.S(t)}
For all $t, \, s\geq0$, the family of operators $S(t)$ satisfies:
\begin{enumerate}
\item $S(t)S(s) = S(t+s)$;
\item $S(t)\, \maK_1 \subset \maK_1$.
\end{enumerate}
\end{lemma}

\begin{proof}
We first notice that $\Df(t) + \Df(s) \circ \delta_t = \Df(t+s)$,
which is easy to check by direct calculation.  By definition, using
also Lemma \ref{lemma.cross}, we have
\begin{multline}
 S(t)S(s) = e^{\Df(t) B} e^{tA} e^{\Df(s) B} e^{sA} = e^{\Df(t) B}
 e^{(\Df(s)\circ \delta_t) B} e^{tA} e^{sA}\\
 = e^{(\Df(t) + \Df(s)\circ \delta_t) B} e^{(t+s)A} = e^{\Df(t+ s) B}
 e^{(t+s)A} = S(t+s).
\end{multline}
This calculation completes the proof of the first part. The last part
  follows from Corollaries \ref{cor.sgA} and \ref{cor.C1}.
\end{proof}

We recall that we assume $\sigma$ is in a bounded interval $I\subset
(0,\infty)$.

\begin{lemma}\label{lemma.unif.limit} We have that for all $j \ge 0$,
 \begin{equation*}
  \|\pa_\sigma^j\Df(t)/t - \sigma^2/2\|_{L^\infty(I)} \to 0 \quad \mbox{as } t \to
  0 \,, \ \ t > 0\,.
 \end{equation*}
\end{lemma}

\begin{proof}
We observe that the function $\pa_\sigma^j\Df(t)/t$, defined on $I \times (0, 1]$,
 extends to a continuous function on $\overline{I} \times [0,
   1]$. Since $I$ is a bounded interval, this fact is enough to provide the result.
\end{proof}

\begin{lemma}\label{lemma.limits}
The following limits in $\maH$ hold:
\begin{enumerate}[(i)]
\item $\lim_{t \searrow 0}S(t) \xi = \xi$ for all
$\xi \in \maH$ and, similarly,
\item  $\lim_{t \searrow 0}t^{-1} (S(t) \xi -
 \xi) = L_0 \xi$ for all $\xi \in \maK_1$.
\end{enumerate}
\end{lemma}

\begin{proof}
By the semigroup property, the operators $e^{tB}$ and $e^{tA}$ are
uniformly bounded if $0 \le t \le \epsilon$, for any fixed $\epsilon >
0$.  Since $I$ is a bounded interval, the functions $\Df(t)$ are
uniformly bounded for $t \le \epsilon$. Moreover,
$\|\Df(t)\|_{L^\infty(I)} \to 0$ as $t \searrow 0$. By the definition
of $S(t)$, Equation \eqref{eq.def.S(t)}, the first part of the lemma
follows.

The second part of the lemma is proved in a similar fashion. Indeed,
the relations $S(t)\, \maK_1 \subset \maK_1$ (see Lemma
\ref{lemma.S(t)}), $\Df'(0) = \sigma^2/2$ (see Lemma
\ref{lemma.unif.limit}), the fact that $e^{tA}$ is a $c_0$
semi-group that leaves $\maK_1$ invariant (Corollary \ref{cor.sgA}),
and Lemma \ref{lemma.cont} give that
\begin{multline*}
 \pa_t \big (T(t)\xi \big )\vert_{t = 0}
 \seq \pa_t \big (e^{\Ef(t) B} e^{tA}\xi \big )\vert_{t = 0}
 \seq \lim_{t \to 0} t^{-1} \big (e^{\Ef(t) B} e^{tA}\xi - \xi \big ) \\
\seq \lim_{t \to 0} t^{-1} \big (e^{\Ef(t) B} e^{tA}\xi - e^{tA}\xi
\big) + \lim_{t \to 0} t^{-1} \big( e^{tA}\xi - \xi \big )
\seq \frac{\partial \Ef}{\pa t}(0) B \xi + A \xi = L_0 \xi\,,
\end{multline*}
whenever $\xi \in \maK_1$.
\end{proof}

We have the following similar result on $\maK_1$:

\begin{lemma}\label{lemma.limits2}
The following limits in $\maK_1$ hold:
\begin{enumerate}[(i)]
\item $\lim_{t \searrow 0}S(t) \xi = \xi$ for all
$\xi \in \maK_1$ and, similarly,
\item  $\lim_{t \searrow 0}t^{-1} (S(t) \xi -
 \xi) = L_0 \xi$ for all $\xi \in \maK_1$ such that $L_0 \xi \in \maK_1$.
\end{enumerate}
These limits are valid also as limits in $\maK_1$ if $\xi \in \maK_1$
in the first limit and if $\xi \in H^{4}(I \times \RR)$ for the second limit.
\end{lemma}

\begin{proof}
 The proof is similar to that of Lemma \ref{lemma.limits2}, but using also
 the second part of Corollary \ref{cor.C1}.
\end{proof}

We finally have that $L_0$ generates the semigroup $S(t)$.

\begin{theorem}\label{thm.main}
Let $\kappa > 0$ and $I= (\alpha, \beta)$, with $0 < \alpha < \theta <
\beta < \infty$, as before.  Then, $S(t) := e^{\Df(t) B} e^{tA}$
defines a $c_0$ semi-group on $\maH$, the generator of which coincides
with $L_0$ on $\maK_1$. Moreover, $S(t)$ defines a $c_0$ semi-group on $\maK_1$.
\end{theorem}

\begin{proof}
The first part is an immediate consequence of Lemmas \ref{lemma.S(t)} and
\ref{lemma.limits}. The second part uses Lemma \ref{lemma.limits2} instead.
\end{proof}

Obtaining explicit formulas is important in practice because it
allows for very fast methods. This is one of the reasons Heston's method
\cite{Heston}
is so popular. Explicit formulas lead also to faster methods in inverse
approaches to the determination of implied volatility, see \cite{Brummelhuis},
for instance.

\begin{corollary}\label{prop.explicit}
Under the assumptions of Theorem \ref{thm.main}, let $h = h(\sigma, x)
\in L^2_{\lambda}(I \times \RR) := e^{\lambda \weight} L^2(I \times
\RR)$ and set $u(t) := S(t) h$. Then, for almost all $\sigma \in I$:
\begin{equation}
  u(t, \sigma, x) \ := \ \frac1{\sqrt{4 \pi \Df}} \, \int \,
  {e^{-\frac{|x-y-\Df|^2}{4\Df}} h(\delta_t(\sigma),y)\, dy} \,
\label{eq:green}
\end{equation}
and $u$ is a mild solution of the Initial Value Problem: \ $\pa_t v -
Lv = 0$, $v(0)=h$. If $h\in \maK$, then $u$ is a strong solution, and
a classical solution provided that $h \in \maC^{1,2}(I\times \RR)\cap
L^2_{\lambda}(I \times \RR)$.
\end{corollary}

\section{Mapping properties and error estimates\label{sec5}}

In this section, we prove mapping properties between weighted spaces
for $e^{tL_0}$, by deriving another formula for its distributional
kernel. We then use these results to compare the semi-groups $S(t) :=
e^{tL_0}$ and $e^{tL}$. We continue to assume that $I = (\alpha,
\beta)$, $0 < \alpha < \theta < \beta < \infty$, and that $\kappa >
0$.

\subsection{Lie algebra identities and semi-groups}

In the previous section we used implicitly commutator estimates
between the operators $A$ and $B$. We collect in the remark below
results pertaining to a general class of operators with properties
similar to the operators $A$ and $B$, which, with abuse of notation,
we continue to denote by $A$ and $B$.

\begin{remark}\label{rem.Hadamard}
 Let $V$ be a finite dimensional space of (usually unbounded)
 operators acting on some Banach space $X$, and let $A$ be a closed
 operator on $X$ with domain $D(A)$.  We make the following
 assumptions
 \begin{enumerate}[(i)] 
  \item All operators in $V$ have the same domain $\maK$, which is
    endowed with a Banach space norm such that, for any $B \in V$, $B
    : \maK \to X$ is continuous.
  \item The space
 \begin{equation*}
  \maW \ede \{ \, \xi \in D(A) \, , \ A\xi \in \maK \, \} \, \cap \,
  \{ \, \xi \in \maK \, , \ B\xi \in D(A) \ (\forall) B \in V\, \}
 \end{equation*}
 is dense in $\maK$ in its induced norm.
 \item If $B \in V$, the closure of the operator $[A, B]$ with domain
   $\maW$ is in $V$.
 \item $A$ generates a $c_0$ semi-group of operators on $X$ that leaves $\maK$
 invariant and induces a $c_0$ semi-group on $\maK$.
 \end{enumerate}
 Then, denoting by $e^{t \ad_A} : V \to V$ the exponential of the
 endomorphism $\ad_A : V \to V$ of the finite dimensional space $V$,
 we obtain the following Hadamard type formula
 \begin{equation}
  e^{tA} B \seq e^{t \ad_A}(B) e^{tA}\,, \quad (\forall) \, B \in V\,.
 \end{equation}
 This relation can be proved by considering the function
 \begin{equation*}
  F(t) \ede e^{tA} B\xi - e^{t \ad_A}(B) e^{tA}\xi\, , \quad B\in V
  \mbox{ and } \xi \in \maW.
 \end{equation*}
 Our assumptions imply that $F(t) \in D(A)$ for all $t$, that $F(t)$
 is differentiable, and that $F'(t) = AF(t)$. By the uniqueness of
 strong solutions to this evolution equation \cite{AmannBook,
   PazyBook}, it follows that $F(t) = 0$ for all $t\geq 0$, since
 $F(0)=0$.
\end{remark}

We shall use the above remark in the following setting.

\begin{remark}\label{rem.Hadamard.commutator}
 Let $V = \CC \pa_\sigma$ with domain $\maK_1$, and let $A := \kappa
 (\theta - \sigma) \pa_\sigma$. consider the adjoint action of $A$ on
 $V$. Since
 \begin{equation*}
  A \pa_\sigma - \pa_\sigma A \seq [A, \pa_\sigma]
  \seq [ \kappa (\theta - \sigma) \pa_\sigma , \pa_\sigma]
  \seq [ \kappa (\theta - \sigma) \pa_\sigma , \pa_\sigma]
  \seq \kappa \pa_{\sigma}\,,
 \end{equation*}
it follows that $e^{tA} \pa_\sigma = e^{\kappa t} \pa_\sigma e^{tA}$.
\end{remark}

In the same spirit, we have the following.

\begin{remark}\label{rem.formula}
We keep the same notation and assumptions as in
   \ref{rem.Hadamard}, but we further assume that $V \simeq \oplus_{a
   \in \RR} V_a$, where
 \begin{equation}
  [A, B_a] \ede AB_a - B_aA \seq a B_a\,, \ \ \mbox{ for any } \
  B_a \in V_a, \ a \in \RR.
 \end{equation}
 Of course, $V_a = 0$, except for finitely many values $a \in
 \RR$. Let $B \in V$ and decompose it as $B = \sum_{a \in \RR} B_a$,
 with $B_a \in V_a$.  We proceed formally to guess a formula for
 $e^{t(A + B)}$. We write $e^{t(A + B)} = e^{tA} e^{\sum_{a}
   f_a(t)B_a}$. Differentiating this inequality, using the semigroup property
 $\partial_t e^{t(A + B)} = (A + B) e^{t(A + B)}$, that $e^{tA} B_a =
 e^{ta} B_a e^{tA}$, and identifying the coefficients, we obtain
 $f_a(t) = (1 - e^{-at})/a = \maE(-at)t$, where $\maE(s) = (e^{s} -
 1)/s$.  Hence, this procedure gives  the (formal!) result
 \begin{equation}
  e^{t(A + B)} \seq e^{tA} e^{\sum_{a} \maE(-at)t B_a}
  \seq e^{\sum_{a} \maE(at)t B_a} e^{tA}\,.
 \end{equation}
 Of course, this procedure has to be justified independently or one
 has to make sense of all the steps in its derivation. In this paper,
 we have chosen to verify independently Formula
 \eqref{eq.def.S(t)}. See also \cite{neebBook}.
\end{remark}

We close by deriving an equivalent formula for $S(T)$, which, by the smoothing
properties of $e^{tB}$, $t>0$,  in $x$, can be used to show that, if $\xi \in
\maC^{1}(I; L^2_{\lambda}(\RR))$, then $u(t) = S(t) \xi$ defines a classical
  solution of $\pa_t u - L_0 u = 0$ for $t>0$. This result uses also
  Corollaries \ref{cor.sgA.new}, \ref{cor.sgA}, and \ref{cor.C1}. The
  method of proof is that of the proof of Lemma \ref{lemma.est.PS(t)}.
For this purpose, we introduce the function:
\begin{equation}\label{eq.def.Cf}
    \Cf(t) \ede \Cf (t, \sigma) \ede
    \frac{(\theta-\sigma)^2}{4\kappa}(e^{2\kappa t} - 1) -
    \frac{\theta(\theta-\sigma)}{\kappa}(e^{\kappa t} - 1) + \frac12
    \theta^2t.
\end{equation}
We notice that $\Cf(t)$ is obtained from $\Df(t)$ by replacing
$\kappa$ with $-\kappa$, so it is still non negative everywhere (see
Proposition \ref{prop.Df}).  Applying the reasoning in the previous
remark, we obtain the following alternative expression for $S(t)$:
\begin{equation}\label{eq.more}
 S(t) \ede e^{\Df(t) B} e^{tA} \seq e^{tA} e^{\Cf(t) B} \,.
\end{equation}

\subsection{Mapping properties}

We shall need certain mapping properties for the semi-groups $e^{tL}$
and $e^{tL_0}$, some of which are standard and some of which we prove in this
subsection.

\begin{lemma}\label{lemma.bound}
 Assume that $I := (\alpha, \beta)$ is bounded and that $\alpha > 0$.
 Then there exists $\epsilon > 0$ such that $\Df(t, \sigma) \ge
 \epsilon t$ for $\sigma \in I$ and $t \in [0, 1]$.
\end{lemma}

\begin{proof}
 Let us consider the function $h(t, \sigma):= \Df(t, \sigma)/t$ for
 $\sigma \in [\alpha, \beta]$ and $t \in (0, 1]$. By Proposition
  \ref{prop.Df}, $h$ extends to a continuous function on $[\alpha,
    \beta] \times [0, 1]$.  By the assumption that $\alpha > 0$ and by
  Proposition \ref{prop.Df}, we have that $h >0$ on $[\alpha, \beta]
  \times [0, 1]$.  Therefore $\epsilon := \inf h > 0$.
\end{proof}

We recall also the following general fact.

\begin{remark}\label{rem.adjoint}
If $T$ generates a $c_0$ semi-group $e^{tT}$ on a Banach space $X$,
then $(e^{tT})^{*}$ will also be a semi-group (but the strong
continuity property may fail). However, if $X$ is {\em reflexive},
then $(e^{tT})^{*}$ is strongly continuous and, in fact,
$(e^{tT})^{*}$ is a $c_0$ semi-group with generator $T^{*}$ (see
Corollary 1.10.6 in \cite{PazyBook}). In other words, $(e^{tT})^{*} =
e^{tT^*}$, if $X$ is reflexive.  Moreover, if $e^{tT}$ is an analytic
semi-group, then $(e^{tT})^{*}$ is also analytic since the function
$(e^{\overline{z}T})^{*}$ is holomorphic in a domain of the form
$\Delta_{\delta}$, $\delta > 0$.
\end{remark}

{\em All the norms $\| \ \|$ below refer to the norm of vectors in
$\maH = L^2_{\lambda}(I \times \RR)$ or of bounded operators on that space.}

\begin{lemma} \label{lemma.est.PS(t)}
Let $s \ge 0$. There exists $C_s > 0$ such that, for all $h \in \maH :=
L^2_{\lambda}(I \times \RR)$,
\begin{equation*}
 t^{s/2} \|e^{\Df(t) B} h\|_{H^{0, s}_{\lambda}(I \times \RR)} \, \le
 \, C_s \|h \| \ede C_s \|h \|_{L^2_{\lambda}(I \times \RR)}\,,
 \ \ \mbox{ for } t \in (0, 1]\,.
\end{equation*}
Consequently, $\|\pa_x^k e^{tL_0}\| \le C t^{-k/2}$, where $t \in (0,
1]$ and $C$ is independent of $t$. In particular, $\pa_x^k e^{tL_0} \xi$
  is continuous in $t$.
\end{lemma}

\begin{proof}
Let us assume first $s = 2n$, for some positive integer $n$.  We have
that the norm $\|g \|_{H^{0, 2n}_{\lambda}(I \times \RR)}$ is
equivalent to the norm $\|g \|\ + \| B^n g \|$, since $B$ is uniformly
strongly elliptic on $\RR$ with totally bounded coefficients (see
Corollary \ref{cor.reg}). In particular, $\|g \|_{H^{0,
    2n}_{\lambda}(I \times \RR)} \le C \big (\|g \| + \| B^n g \| \big
)$.  It is therefore enough to show that there exists $C_s'$ such that
\begin{equation}\label{eq.enough}
  \| e^{\Df(t) B} h \| + \| B^n e^{\Df(t) B} h \| \, \le \, C_s'
  t^{-n} \|h \|\,,
\end{equation}
since then the desired relation follows with $C_s = C C_s'$.  Lemma
\ref{lemma.bound} gives
\begin{multline*}
  \| e^{\Df(t) B} h \| + \| B^n e^{\Df(t) B} h \|
  \seq \| e^{\Df(t) B} h \| + \| e^{(\Df(t) - \epsilon t)B } B^n
  e^{\epsilon t B} h \| \\
  \le \, C \big (\| h \| + \| B^n e^{\epsilon t B} h \| \big )
 \le \, C (\epsilon t)^{-n} \|h \|\,,
\end{multline*}
since $e^{g B}$ is bounded on $L^2_{\lambda}(I \times \RR)$, if $g \ge
0$ is bounded measurable, and $t^{n}B^n e^{tB}$ is also bounded on the
same space (by Equation \eqref{eq.regularity} for $T = B$). Here, we
have used the assumption that $I$ is bounded. This argument proves
Equation \eqref{eq.enough}, and consequently also the result for $s =
2n$. For general $s \ge 0$, the result follows by complex
interpolation.

To prove the last part, we write
\begin{equation*}
  \pa_x^{2k} e^{tL_0} \seq \pa_x^{2k}(\mu_0 - B)^{-k} (\mu_0 -
  B)^{k}e^{\Df(t) B} e^{tA}\,,
\end{equation*}
where $\mu_0$ is large.  We have that $\pa_x^{2k}(\mu_0 - B)^{-k}$ is
bounded by the uniform strong ellipticity of $B$ and Theorem
\ref{thm.reg}.  Remark \ref{rem.a.est} and Lemmas \ref{lemma.cont} and
\ref{lemma.bound} show that $(\mu_0 - B)^{k}e^{\Df(t) B}$ depends
smoothly on $t$. Next, Remark~\ref{rem.a.est} also gives that
$\|(\mu_0 - B)^{k}e^{\Df(t) B}\| \le C t^{-k}$. This implies that
$\|\pa_x^{2k} e^{tL_0}\| \le C t^{-k}$. Our desired estimate
$\|\pa_x^k e^{tL_0}\| \le C t^{-k/2}$ is then obtained by
interpolation. Finally, using also Lemma~\ref{lemma.cont}, we obtain
that $\pa_x^{k} e^{tL_0}$ depends continuously on $t$.
\end{proof}

In the same way, we obtain the following result.

\begin{lemma} \label{lemma.est.tL}
If $h \in \maH := L^2_{\lambda}(I \times \RR)$, then
\begin{equation*}
 \|e^{tL} h\|_{H^{s}_{\lambda}(I \times \RR)} \, \le \, C t^{-s/2} \|h \|\,.
\end{equation*}
If $P$ is a differential operator of order $k$ with totally bounded
coefficients on $I \times \RR$, then $P e^{tL}$ and $e^{tL} P$ extend
to bounded operators on $\maH$ of norm $\le C t^{-k/2}$ that depend
smoothly on $t > 0$.
\end{lemma}

\begin{proof}
Since $L$ is uniformly strongly elliptic with totally bounded
coefficients, there exists $\mu_0 > 0$ such that
\begin{equation*}
  L -\mu_0: H^{m+1}_{\lambda}(I \times \RR) \cap \{u(\alpha, x) =
  u(\beta, x) = 0\} \to H^{m-1}_{\lambda}(I \times \RR),
\end{equation*}
is an isomorphism by Theorem \ref{thm.reg} and Corollary
\ref{cor.Garding}.  Let $(L - \mu_0)^{-1}$ denote the resulting map
$L^{2}_{\lambda}(I \times \RR) \to H^{2}_{\lambda}(I \times
\RR)$. Then $(L - \mu_0)^{-1}$ maps $H^{m-1}_{\lambda}(I \times \RR)
\to H^{m+1}_{\lambda}(I \times \RR)$ continuously. In particular, $(L
- \mu_0)^{-n} : L^{2}_{\lambda}(I \times \RR) \to H^{2n}_{\lambda}(I
\times \RR)$ is continuous. Let us assume now that $s = 2n$.  Then
\begin{multline*}
  \|e^{tL} h\|_{H^{2n}_{\lambda}(I \times \RR)} \seq \|(L -
  \mu_0)^{-n} (L - \mu_0)^{n} e^{tL} h\|_{H^{2n}_{\lambda}(I \times
    \RR)} \\
  \le \, C \|(L - \mu_0)^{n} e^{tL} h\|_{L^{2}_{\lambda}(I \times
    \RR)} \, \le \, C t^{-s/2} \|h \|\,,
\end{multline*}
since $L$ generates an analytic semi-group. For general $s$, the
inequality follows by interpolation.

Let $P$ now be as in the statement of the lemma. Then $P :
H^{k}_{\lambda}(I \times \RR) \to L^{2}_{\lambda}(I \times \RR)$ is
bounded. This implies the result for $P e^{tL}$.  The result for
$e^{tL}P$ is obtained by taking adjoints, since $L^{*}$ is uniformly
strongly elliptic with totally bounded coefficients and generates an
analytic semi-group.
\end{proof}

Lemma \ref{lemma.est.tL} gives the following result. All norms of
operators are on $L^2_{\lambda}(I \times \RR)$.

\begin{lemma}\label{lemma.four}
The operator $F(s) \ede e^{(t-s)L} \pa_\sigma e^{sL}$ extends, for
each $s \in [0, t]$, to a bounded operator on $L^2_{\lambda}(I \times
\RR)$, and the resulting function is continuous in $s \in [0, t]$ and
differentiable for $s \in (0, t)$.  Its derivative is the function
\begin{equation*}
 F'(s) \seq e^{(t-s)L} [\pa_\sigma, L] e^{sL} \,,
\end{equation*}
which satisfies $\|F'(s)\| \le C t^{-1}$, with $C$
independent of $ 0 < s < t \le 1$.
\end{lemma}

\begin{proof}
Lemma \ref{lemma.est.tL} gives that both functions $e^{(t-s)L}$ and
$\pa_\sigma e^{sL}$ are continuous on $(0,T]$ and infinitely many
times differentiable on $ (0,t)$ as functions with values in the space
of bounded operators.  The formula for the derivative follow from the
standard formula $(e^{sL})' = Le^{sL}$, which we note to be valid in
norm, since $L$ generates an analytic semi-group and $s > 0$. The
continuity on $[0, t)$ follows in the same way by considering
  $e^{(t-s)L}\pa_\sigma$ and $ e^{sL}$.

If $s \le t/2$, since $[\pa_\sigma, L]$ is a second order differential
operator, Lemma \ref{lemma.est.tL} implies that $e^{(t-s)L}
[\pa_\sigma, L]$ is bounded with norm $\le C(t-s)^{-1} \le 2C
t^{-1}$. In addition, $\|F'(s)\| \le Ct^{-1}$ given that $e^{sL}$ is
norm bounded.  The case $s \ge t/2$ is completely similar, using the
bounds for $[\pa_\sigma, L] e^{sL}$ provided by Lemma
\ref{lemma.est.tL}.
\end{proof}

\subsection{A comparison of $e^{tL}$ and $e^{tL_0}$}

In this last section, we compare the semi-groups $S(t) := e^{tL_0}$
and $e^{tL}$. We recall that we set $L = L_0 + V$, where $V = \nu L_1
+ \nu^{2} L_2 = \nu \rho\sigma^2\pa_x\pa_\sigma + \frac{\nu^{2}
  \sigma^{2}}{2} \pa_\sigma^2$, and we think of $L$ as a perturbation
of $L_0$ for $\nu$ sufficiently small. We recall also that $\maK_1 :=
H^2_{\lambda}(I \times \RR)$ and $\maK_0 := H^2_{\lambda}(I \times
\RR) \cap \{ u(\alpha, x) = u(\beta, x) = 0\}$, where $I = (\alpha,
\beta)$ is a fixed bounded interval containing $\theta$.

The approach presented in this subsection can be iterated to derive higher-order
approximate solutions in the parameter $\nu$. These are the focus of current
work by the authors.

\begin{lemma}\label{lemma.cont2}
 Let $\xi \in \maK_1$. Then $F(s) \ede e^{(t-s)L}e^{sL_0}\xi$ is
 continuous on $[0, t]$ and differentiable on $(0, t)$, with $F'(s)
 \seq - e^{(t-s)L} V e^{sL_0}\xi$.
\end{lemma}

\begin{proof}
Since $\xi$ is in the domain of $L_0$ (which contains $\maK_1$, by
Theorem \ref{thm.main}), the function $\zeta(s) := e^{sL_0}\xi$ is
differentiable for $s \ge 0$.  But $e^{tL}$ is a $c_0$ semi-group,
therefore Lemma \ref{lemma.cont} gives that $F(s) = e^{(t-s)L}
\zeta(s)$ is continuous on $[0, t]$. Since $e^{tL}$ is an analytic
semi-group, it follows in addition that $F(s)$ is differentiable for
$s \in (0, t)$, by Lemma \ref{lemma.diff}, and its derivative is
$F'(s) \seq - e^{(t-s)L} V e^{sL_0}\xi$.

\end{proof}

We continue to assume that $\| \, \cdot \, \|$ refers to the norm in $\maH =
L^2_{\lambda}(I \times \RR)$ or the operator norm of bounded operators
on $\maH$.

\begin{lemma}\label{lemma.sigma.x}
Let $\xi \in \maK_1$, then $e^{(t-s)L} L_1 e^{sL_0} \xi$ depends
continuously on $s$ and
\begin{equation*}
  (\rho \nu)^{-1} \|e^{(t-s)L} L_1 e^{sL_0} \xi \| \seq \|e^{(t-s)L}
  \sigma^2 \pa_\sigma \pa_x e^{sL_0} \xi \| \, \leq \, C(t-s)^{-1/2}
  s^{-1/2} \|\xi\|\,.
\end{equation*}
Consequently, $\Big \| \int_{0}^{t} e^{(t-s)L} L_1 e^{sL_0} \, ds \Big
\| \le C \rho \nu $.
\end{lemma}

\begin{proof}
Lemmas \ref{lemma.est.PS(t)} and \ref{lemma.est.tL} show that
$e^{(t-s)L} \sigma^2 \pa_\sigma$ and $\pa_x e^{s(L_0-\kappa)}\xi$
satisfy the assumptions of Lemma \ref{lemma.cont}, so $e^{(t-s)L}
\sigma^2 \pa_\sigma\pa_x e^{s(L_0-\kappa)}\xi$ is continuous in $s$.
Similarly, Lemmas \ref{lemma.est.PS(t)} and \ref{lemma.est.tL} give
\begin{equation*}
 \|e^{(t-s)L} \sigma^2 \pa_\sigma \pa_x e^{s(L_0-\kappa)} \xi \| \le
 \|e^{(t-s)L} \sigma^2 \pa_\sigma \| \, \| \pa_x e^{s(L_0-\kappa)}
 \xi\| \le C (t-s)^{-1/2}s^{-1/2}\|\xi\| .
\end{equation*}
The integral can be estimated by splitting the interval $[0, t]$ in
two halves.
\end{proof}

To estimate the terms involving $L_2$, we exploit the next result.

\begin{lemma}\label{lemma.corrected} Let $\xi \in \maK_1$, then
$\pa_\sigma e^{tL_0}\xi = e^{t(L_0 - \kappa)}\pa_\sigma \xi +
  \frac{\pa \Df(t, \sigma)}{ \pa \sigma} B e^{tL_0} \xi.$
\end{lemma}

\begin{proof} The main calculation is contained in Remark
\ref{rem.Hadamard.commutator}. More precisely, this is a direct
calculation using Equation \eqref{eq.def.S(t)}, together with Lemma
\ref{lemma.diff}, with Hadamard's theorem (see Remarks
\ref{rem.Hadamard} and \ref{rem.Hadamard.commutator}), and with the
fact that $\adj_{L_0}(\pa_\sigma)\adj_A(\pa_\sigma) = \kappa
\pa_\sigma$.
\end{proof}

However, the terms in $L_2$ present some additional challenges, since
$L_0$ is not elliptic.

\begin{lemma}\label{lemma.sigma.sigma}
Let $\xi \in \maK_1$, then $e^{(t-s)L} L_2 e^{sL_0}\xi$ depends
continuously on $s$ and the following estimate holds:
\begin{equation*}
  \frac{2}{\nu^2} \| e^{(t-s)L} L_2 e^{sL_0}\xi \| \seq \| e^{(t-s)L}
  \sigma^2 \pa_\sigma^2 e^{sL_0} \xi \| \, \le \, C (t-s)^{-1/2} (\big
  \|\pa_\sigma \xi\| + \|\xi\| \big) \,.
\end{equation*}
Consequently, $\Big \| \int_{0}^{t} e^{(t-s)L} L_2 e^{sL_0} \xi \, ds
\Big \| \le C \nu^2 \sqrt{t} \big (\|\pa_\sigma \xi\| + \|\xi\|
\big)$.
\end{lemma}

\begin{proof}
Lemma \ref{lemma.corrected} gives
\begin{equation}\label{eq.one.bis}
  e^{(t-s)L} \sigma^2 \pa_\sigma^2 e^{sL_0} \xi \seq e^{(t-s)L}
  \sigma^2 \pa_\sigma \Big ( e^{s (L_0 - \kappa)} \pa_\sigma \xi +
  \frac{\pa \Df(s, \sigma)}{ \pa \sigma} B e^{sL_0} \xi \, \Big ) \,.
\end{equation}

As in the proof of Lemma \ref{lemma.sigma.x}, Lemmas
\ref{lemma.est.tL} and \ref{lemma.est.PS(t)} give that both
$e^{(t-s)L} \sigma^2 \pa_\sigma e^{sL_0}$ and $e^{(t-s)L} \sigma^2
\pa_\sigma \frac{\pa \Df}{ \pa \sigma} B e^{sL_0}$ define bounded
operators that depend continuously on $s \in (0, t)$ in the strong operator
topology. We
estimate separately the norm of each of them. Again from Lemma
\ref{lemma.est.tL}, we obtain
\begin{equation*}
 \|e^{(t-s)L} \sigma^2 \pa_\sigma e^{s(L_0-\kappa)}\| \, \le \,
 \|e^{(t-s)L} \sigma^2 \pa_\sigma\| \, \| e^{s(L_0-\kappa)}\| \, \le
 \, C (t-s)^{-1/2} \,.
\end{equation*}
For the estimate of the second term, we first notice that $\|\frac{\pa
  \Df(t, \sigma)}{ \pa \sigma}\|_{L^\infty(I)} \le C t$, since
  the function $\frac{\pa \Df(t, \sigma)}{ t \pa \sigma}$ extends to a
  continuous function on $\overline{I} \times [0, 1]$.  Hence,
$\|\frac{\pa \Df(s, \sigma)}{ \pa \sigma} B e^{sL_0} \| \le \|s B
e^{sL_0} \| \le C$ by Lemma \ref{lemma.est.PS(t)}, and
\begin{multline*}
 \Big \|e^{(t-s)L} \sigma^2 \pa_\sigma \frac{\pa \Df(s, \sigma)}{ \pa
   \sigma} B e^{sL_0} \Big \| \, \leq \, \|e^{(t-s)L} \sigma^2
 \pa_\sigma \| \, \Big \| \frac{\pa \Df(s, \sigma)}{ \pa \sigma} B
 e^{sL_0} \Big \| \le C(t-s)^{-1/2} \,.
\end{multline*}
The last two displayed equations and Equation \eqref{eq.one.bis} then
combine to give the first part of the statement. The last relation in
the statement follows directly by integrating the first one.
\end{proof}

Combining the previous two lemmas we obtain the following corollary.

\begin{corollary}\label{cor.est.F'} The family
$G(s) := e^{(t-s)L} V e^{sL_0}$ consists of bounded operators on
  $\maH$.  Moreover, for any $\xi \in \maK_1$, $G(s)\xi$ is continuous
  and integrable in $s \in (0, t)$ and we have:
\begin{equation*}
  \Big \| \int_0^t G(s)\xi \, ds \Big \|
  \ede \Big \|\int_0^t e^{(t-s)L} V e^{sL_0}\xi \, ds \Big \| \, \leq \, C \Big (
  \rho \nu \|\xi\| + \nu^2 \sqrt{t} \big (\|\pa_\sigma \xi\| + \|\xi\|
  \big) \Big )\,.
\end{equation*}
\end{corollary}

Lemma \ref{lemma.cont2} and Corollary \ref{cor.est.F'} then give:
\begin{equation*}
 e^{tL} \xi - e^{tL_0}\xi \seq F(0) - F(t) = \int_{0}^{t} e^{(t-s)L} V
 e^{sL_0}\xi \, ds \,.
\end{equation*}
The final estimate is for $\xi \in H^1(I, L^2_{\lambda}(\RR)) :=
  \{\zeta \in L^2_{\lambda}(I \times \RR), \ \pa_\sigma \zeta \in
  L^2_{\lambda}(I \times \RR)\}$.

\begin{theorem}\label{thm.estimate}
 There is $C > 0$ such that
 \begin{equation*}
  \| e^{tL}\xi - e^{tL_0}\xi \| \ \le \ C \, \nu \big ( \, \|\xi\| +
  \nu \|\pa_\sigma \xi \| \, \big )\,,
 \end{equation*}
for $\xi \in H^1(I, L^2_{\lambda}(\RR))$ and $0 \le t \le
  T$. The bound $C$ depends on $T$, but not on $\xi$.
\end{theorem}

\begin{proof}
 The statement was proved for $\xi \in \maK_1$. For general $\xi$, it
   follows from the density of $\maK_1 := H^{2}_{\lambda}(I\times
   \RR)$ in $H^1(I, L^2_{\lambda}(\RR))$ and the continuity on $H^1(I,
   L^2_{\lambda}(\RR))$ of all the operators appearing on the left and
right sides of the inequality.
\end{proof}

We close by observing that similar commutator estimates were obtained in
\cite{Wen1, Wen2, CCMN, GHN}. The main difficulty addressed in this work  is
that $L_0$ is not an elliptic operator.

\bibliographystyle{plain}

\end{document}